\documentclass[conference,letterpaper]{IEEEtran}

\usepackage[thmmarks]{ntheorem}
\usepackage[utf8]{inputenc} 
\usepackage[T1]{fontenc}
\usepackage{url}
\usepackage{ifthen}
\usepackage{cite}
\usepackage[cmex10]{amsmath} 

\usepackage{amssymb}
\usepackage{ntheorem}
\usepackage{graphicx}
\usepackage{textcomp}
\usepackage{dsfont}
\DeclareMathOperator*{\argmin}{arg\,min}
\DeclareMathOperator*{\argmax}{arg\,max}
\DeclareMathOperator{\diag}{diag}

\usepackage{algorithm} 
\usepackage{algpseudocode} 

\begin{document}
\title{The Safety-Privacy Tradeoff in Linear Bandits} 


\author{%
  \IEEEauthorblockN{Arghavan Zibaie and Spencer Hutchinson and Ramtin Pedarsani and Mahnoosh Alizadeh}
  \IEEEauthorblockA{
                    University of California Santa Barbara\\
                    \{zibaie,shutchinson,ramtin,alizadeh\}@ucsb.edu}


}

\maketitle


\begin{abstract}
We consider a collection of linear stochastic bandit problems, each modeling the random response of different agents to proposed interventions, coupled together by a global safety constraint. We assume a central coordinator must choose actions to play on each bandit with the objective of regret minimization, while also ensuring that the expected response of all agents satisfies the global safety constraints at each round, in spite of uncertainty about the bandits' parameters. The agents consider their observed responses to be private and in order to protect their sensitive information, the data sharing with the central coordinator is performed under local differential privacy (LDP). However, providing higher level of privacy to different agents would have consequences in terms of safety and regret. We formalize these tradeoffs by building on the notion of the sharpness of the safety set - a measure of how the geometric properties of the safe set affects the growth of regret - and propose a unilaterally unimprovable vector of privacy levels for different agents given a maximum regret budget.

\end{abstract}

\theoremseparator{.}
\newtheorem{proposition}{Proposition}
\newtheorem{theorem}{Theorem}
\newtheorem{Corollary}{Corollary}
\newtheorem{lemma}{Lemma}
\newtheorem{Fact}{Fact}
\newtheorem{remark}{Remark}
\newtheorem{assumption}{Assumption}
\newtheorem{definition}{Definition}

\newcommand{\eqdef}{\vcentcolon=}
\newcommand{\beq}{\begin{equation}}
\newcommand{\eeq}{\end{equation}}
\newcommand{\ie}{i.e., }

\section{INTRODUCTION}

The stochastic linear bandit problem constitutes a sequential decision-making problem wherein, at each round, we observe a response to chosen actions which is a perturbed linear function of the action parameterized by an unknown parameter vector. When applying this tool to safety-critical applications such as health care \cite{ijcai2017p389}, power systems and transportation \cite{10457043,chaudhary2022safe}, the decision-making tasks must operate within certain safety constraints that depend on the unknown bandit parameters, and violations of these constraints can result in adverse events. For example, when sequentially learning how to price electricity to manage the demand of users whose price response is unknown, the price setting entity must ensure that the resulting demands do not violate electric distribution system constraints from day one, in spite of its uncertainty about user responses. As a result, variants of the linear stochastic bandit problem with constraints have been studied in the literature, e.g., \cite{amani2019linear,pmlr-v211-hutchinson23a,pacchiano2021stochastic}. 

In many such safety-critical applications, however, a central challenge lies in the fact that users might consider their responses to the central coordinator's interventions to be private information. This can further complicate the task of learning optimal interventions while keeping the system safe at all rounds of the learning process. To formalize this challenge mathematically, we consider a collection of linear stochastic bandit problems whose  responses are tied together through a global safety constraint, coupling what actions can be safely played on each bandit. A central coordinator must choose these actions by observing data on the agents' responses to its past proposed actions. The agents consider their observed response to be sensitive information that they do not want to freely share with the coordinator and as such, a privacy preservation mechanism is needed to allow this data exchange. One form of commonly used privacy guarantee in control systems is Differential Privacy (DP), first introduced in \cite{dwork2006differential,dwork2014algorithmic}. DP guarantees that the probability of any output sequence remains unaffected by the inclusion or exclusion of any single database entry. Local Differential Privacy (LDP), a variant focusing on individual data source privacy before aggregation, was proposed in  \cite{ye2020local},\cite{wang2021local}, allowing each agent to guarantee its privacy locally before sharing their data. Subsequent studies have considered the implications of DP (e.g., \cite{shariff2018differentially, mishra2015nearly, tossou2016algorithms, hanna2024differentially}) and LDP (e.g., \cite{han2021generalized, tao2022optimal ,zheng2020locally,chen2020locally}) within bandit frameworks, highlighting the tradeoff between privacy and performance.

\textit{Contribution}: While previous works deal with the challenge of ensuring \textit{either} safety \textit{or} privacy separately, in a multi agent safety-critical system, the privacy requirements of each agent might have different implications on regret. This is because the difficulty of learning the private parameters of various agents would vary based on how their parameters affect the global system's safety constraints. This paper is focused on formalizing this \emph{privacy-regret tradeoff} in safe learning by building on the notion of \emph{sharpness} introduced in \cite{pmlr-v211-hutchinson23a}, which measures how the geometric properties of the safe action set affect regret, without considering privacy. To study the effect of privacy in this context, we first propose a safe and private linear bandit algorithm which generalizes the Safe-LUCB method of \cite{amani2019linear} to consider agents that apply LDP protocols to their observed responses before sharing it with a central coordinator. The main effect of LDP is that the coordinator has lower confidence about each agent's  parameter, which leads to additional regret. We provide a general regret bound for any polytopic safe set, and  then determine the exact sharpness for the special class of the safe set being a simplex. Then for a specified \textit{regret budget}, we suggest a set of unimprovable privacy levels for different agents.

\subsection{Related Works}
We briefly discuss some related previous works. 
\begin{itemize}
        \item \emph{Bandit problems under safety critical constraints:} In many bandit problems, the decision maker should choose actions that satisfy some constraints.  For example, in \cite{amani2019linear} the decision maker ensures  stage-wise safety, introducing the Safe-LUCB algorithm. \cite{Khezeli_Bitar_2020} guarantees that the observed reward is not below a safe threshold. In \cite{pmlr-v211-hutchinson23a} the notion of sharpness was introduced on a class of safety constraint sets for bandit problems, which studied the cost of safety given the geometry of the constraint set.
        \item \emph{Bandit algorithms with DP and LDP guarantee:} The notion of differential privacy has received attention in various bandit problems. For example, \cite{shariff2018differentially} introduces a notion called joint differential privacy on contextual linear bandits using tree-based algorithms to guarantee the privacy. In \cite{dubey2020differentially}, the DP guarantee was studied for both centralized and decentralized federated contextual linear bandits. In \cite{10423210}, the problem of distributed linear bandit with partial feedback under differential privacy guarantee is studied. Local Differential Privacy (LDP) in bandit problems was studied in \cite{zheng2020locally}, proposing a black box framework which can be adopted by various types of context-free bandit algorithms. 
        In \cite{han2021generalized}, a stochastic gradient descent based algorithm for contextual linear bandit was introduced ensuring LDP protection for both the context and the observed response of the user.
\end{itemize}
\subsection{Notations}
We define $ [T] = \{1,2,\cdots,T\}$. The transpose of a column vector $\mathbf{x} \in \mathbb{R}^d$ is shown by $\mathbf{x}^T = [x_1,x_2,\cdots,x_d]$ where $x_i$ is the i'th component of $\mathbf{x}$ and its Euclidean norm is $||\mathbf{x}||_2$ and for a positive definite $d \times d$ matrix $\mathbf{A}$, $||\mathbf{x}||_{\mathbf{A}} = \sqrt{\mathbf{x}^T\mathbf{A}\mathbf{x}}$. 
The spectral norm of matrix $\mathbf{M}$ is denoted by $||\mathbf{M}||_2$. The covariance matrix of vector $\mathbf{X}$ showed by $ \Sigma(\mathbf{X}) = \mathbb{E}[\mathbf{X}\mathbf{X}^T]$ and its minimum eigenvalue is $\check{\lambda}(\Sigma(\mathbf{X})) \ge 0$. A closed ball is defined with respect to norm $||.||$ as $\Bar{\mathcal{B}}_{||.||}(r) := \{ \mathbf{x}\in \mathbb{R}^M : ||\mathbf{x}|| \leq r \}$ while the open ball is shown by $\mathcal{B}_{||.||}(r)$, where r is the radius of these balls. Also $\mathbf{x}_{1:t}$ is stands for all the elements of the set $\{ \mathbf{x}_1 , \mathbf{x}_2 ,\cdots, \mathbf{x}_t\}$.  

\section{PROBLEM SETUP}
We consider $M$ stochastic linear bandit problems coupled by a global safety constraint that needs to be satisfied at all rounds $t \in [T]$. Each bandit problem can be viewed as an autonomous entity (a.k.a., agent), receiving an action (input) and observing a noisy response (output) at each round. A central coordinator is the decision-maker that simultaneously interacts with all of these bandits, selecting an action for each at every round. Specifically, at each round $t$, a central coordinator chooses an action $\mathbf{x}_{t,m}$ for each bandit $m \in [M]$, respectively from a given compact action set $\mathcal{D}_{m} \subset \mathbb{R}^d$. 
Upon playing action $\mathbf{x}_{t,m}$, agent $m$ observes the noisy response:\begin{equation} \label{observed} y_{t,m} =  \theta_{*,m}^{T}\mathbf{x}_{t,m} + \eta_{t,m},\end{equation} which is the sum of a linear function of the action $\mathbf{x}_{t,m}$ parameterized by an unknown  vector  $\theta_{*,m}$, plus noise $\eta_{t,m}$.

The central modeling assumption of this work is that the bandits  share a global safety constraint. As such, the coordinator’s actions for each one are not chosen independently but must collectively remain in a \textit{globally safe} region.
To mathematically write the global safety constraint that couples the bandit problems $m \in [M]$, we first define some  notation. We define the block-diagonal matrix $\Theta_*$ by placing the row vectors 
$\theta_{*,1}^T, \theta_{*,2}^T, \ldots, \theta_{*,M}^T$ along the diagonal 
blocks:
\[
\Theta_*
= 
\begin{bmatrix}
\theta_{*,1}^T & \mathbf{0}      & \cdots & \mathbf{0}      \\
\mathbf{0}   & \theta_{*,2}^T    & \cdots & \mathbf{0}      \\
\vdots       & \vdots          & \ddots & \vdots          \\
\mathbf{0}   & \mathbf{0}      & \cdots & \theta_{*,M}^T
\end{bmatrix} \in \mathbb{R}^{M \times (Md)},
\]
which we will denote succinctly by 
$\Theta_* = \theta_{*,1}^T \oplus \theta_{*,2}^T \oplus \cdots \oplus \theta_{*,M}^T.$ Furthermore, denote the aggregate decision set as $\mathcal{D}:= \{ \mathbf{X}_t : \mathbf{x}_{t,m} \in \mathcal{D}_{m} \, \forall m \in [M]\}$, where  $\mathbf{X}_t = [\mathbf{x}_{t,1}^T , \mathbf{x}_{t,2}^T , \cdots , \mathbf{x}_{t,M}^T ]^T$
is the concatenated vector of chosen actions at time $t$. 

The following global
safety constraint has to be satisfied at each round $t$:
\begin{equation}\label{safety}
    \Theta_* \mathbf{X}_t \in \mathcal{Y},
\end{equation}
where $\mathcal{Y} \subset \mathbb{R}^M $ is a compact set. Note that given the fact that $\Theta_*$ is unknown, ensuring this requires careful analysis.

\textbf{Central Coordinator's Objective}: The   coordinator wishes to maximize $\sum_{t=1}^T f(\Theta_{*}\mathbf{X}_t)$ for a given function $f: \mathbb{R}^M \rightarrow \mathbb{R}$ of the expected response vector $\Theta_{*}\mathbf{X}_t$, while ensuring that the safety constraint \eqref{safety} is satisfied. However, given the fact that $\Theta_{*}$ is unknown, this goal is instead defined as minimizing the so-called cumulative pseudo-regret $R_T$ over the   $T$ rounds by choosing actions $\mathbf{X}_{t}$, given by: 
\begin{equation}\label{regret}
   R_T =  \sum_{t=1}^T (f(\Theta_{*}\mathbf{X}_*) - f(\Theta_{*}\mathbf{X}_t)).
\end{equation}
Here,  $\mathbf{X}_* = \underset{\Theta_* \mathbf{X} \in \mathcal{Y}}{\argmax}\,f(\Theta_{*}\mathbf{X})$ is the optimal safe action. For brevity, we will refer to cumulative pseudo-regret   as regret. 
 
\textbf{User Privacy}:  The above task is further complicated by the privacy requirements of the agents. Specifically, each agent $m \in [M]$ considers the observed responses $y_{t,m}$ in \eqref{observed} as private information that they do not wish to directly share with the central coordinator (or   they do not trust that the central coordinator will not reveal them to other agents). As such, they are not willing to reveal their observed response $y_{t,m}$ to the central coordinator without a privacy-preservation guarantee. Note that this privacy requirement only applies to the epoch-wise response  $y_{t,m}$ and should not be confused with keeping the unknown bandit parameter $\theta_{*,m}$ private.  

The provided privacy guarantee in this work is based on the notion of Local Differential Privacy (LDP), defined next as a general mechanism on arbitrary data points.

\begin{definition}\label{LDP_def}
    \textbf{(Local Differential Privacy)} A randomized mechanism $\mathcal{A}: \mathcal{Y}\rightarrow\mathcal{U}$, is said to be $(\epsilon,\delta)$-LDP , if for any two input $ y , y' \in \mathcal{Y}$ and for any (measurable) subset $\mathbf{U} \subset \mathcal{U}$, the following inequality holds: 
\begin{equation}
    \mathbf{P}[\mathcal{A}(y) \in \mathbf{U}] \leq e^\epsilon \,\mathbf{P}[\mathcal{A}(y') \in \mathbf{U}] + \delta.
\end{equation}
\end{definition}

LDP is a privacy model that ensures the protection of individual data points by applying a randomized mechanism to each data point before any analysis   \cite{dwork2014algorithmic}\cite{ye2020local}. This means each agent (data point provider) independently applies a privacy-preserving mechanism to their data locally before sharing it. The mechanism is designed such that it makes it hard for adversaries (including the central coordinator) to infer the original value of any specific data point, even if they have unlimited computational resources.
Here (\(\epsilon\),\(\delta\)) serve as controls for the privacy guarantee:

\begin{itemize}
  \item \(\epsilon\)  quantifies the privacy level, with a smaller \(\epsilon\) indicating stronger privacy but potentially less accuracy.
  \item \(\delta\)  allows for a small probability of the mechanism's failure to meet the \(\epsilon\)-privacy guarantee. A purely \(\epsilon\)-LDP mechanism, where \(\delta = 0\), is stronger than one with \(\delta > 0\).
\end{itemize}

One of the commonly used mechanisms which provides an LDP guarantee is the Gaussian mechanism \cite{dwork2014algorithmic}. 

\begin{lemma}\label{LDP_M}For any $y:\mathcal{D} \rightarrow \mathcal{Y} \subset \mathbb{R}^d\,$,  let \newline
$\sigma_{\epsilon,\delta} = \frac{1}{\epsilon}\,\sup_{\mathbf{x,x'}\in\mathcal{D}}||y(\mathbf{x}) - y(\mathbf{x}')||_2\sqrt{2ln(\frac{1.25}{\delta})}$.
The Gaussian mechanism, which adds iid noise drawn from distribution $\mathcal{N}(0,\sigma_{\epsilon,\delta}^2I_d)$ to each $y$ ensures $(\epsilon,\delta)$-LDP protection on $y$.
\end{lemma}

  The question we study is whether we can carefully tailor the privacy levels afforded to   agents while  maintaining a given regret bound. This allows us to \textit{quantify  the  trade-off between individual privacy needs and important collective operational objectives such as safety}. To allow this study, we denote the privacy parameter for bandit \(m\)  as \((\epsilon_m,\delta_m)\). Consequently, in each round, rather than directly sharing the observed response \(y_{t,m}\), each agent independently samples a noise \(h_{t,m} \sim \mathcal{N}(0,\sigma_{m}^2)\) and adds it to the observed response \(y_{t,m}\) to create a protected version $u_{t,m} = y_{t,m} + h_{t,m}$.

\subsection{Assumptions}
Let $\mathcal{F}_{t-1} = \sigma(\mathbf{x}_{1:t},\eta_{1:t-1 , 1:M},h_{1:t-1 , 1:M})$ be the $\sigma$-algebra at round t. The following standard assumptions will be used.

\begin{assumption}
\label{A1}
(Sub-Gaussian Noise) For  $\forall m \in [M]$ and $\forall t \in [T]$,  $\eta_{t,m}$ is conditionally zero-mean R-sub-Gaussian with a fixed constant $R \geq 0 $, i.e., $\mathbb{E}[\eta_{t,m}|\mathcal{F}_{t-1}] = 0$ and $\mathbb{E}[e^{\lambda\eta_{t,m}}|\mathcal{F}_{t-1}] \leq exp(\lambda^2\mathit{R}^2/2),  \,\,\, \forall \lambda \in \mathbb{R}$.
\end{assumption}

\begin{assumption}
\label{A3}
(Lipschitz) For all $y,y'\,\, in\,\,\mathcal{Y} $, the function f is $L-lipschitz $ on $ \mathcal{Y} $ such that $|f(y)-f(y')|\leq L||y-y'||_2$.
\end{assumption}

\begin{assumption}
\label{A2}
(Boundedness) There exist positive constants $S$ and $K$ such that  $\forall m \in [M]$ and $\forall t \in [T]$,  $||\theta_{*,m}||_2\leq S$ and $||\mathbf{x}_m||_2\leq K$ for all $\mathbf{x}_m \in \mathcal{D}_{m}$. Also $y_{t,m} \in [-1,1]$. 
\end{assumption}

For brevity of presentation, we also assume in the following that the decision set $\mathcal{D}$ is not restrictive.

\begin{assumption}
\label{A5}
The safe action set $\mathcal{D}_s$ is a subset of $\mathcal{D}$:
\begin{equation}\label{ds}
    \mathcal{D}_s := \{\mathbf{X} : \Theta_*\mathbf{X} \in \mathcal{Y}\} \subseteq \mathcal{D}
\end{equation}
\end{assumption}

It follows from Assumption \ref{A2} that $\Theta_*$ is initially known to be in the set,
\begin{equation*}
    \mathcal{Q} := \{\Theta \in \mathbb{R}^{M \times Md} : ||\theta_m||_2 \leq S \, , \forall m \in [M]\},
\end{equation*}
where we consider an aggregate parameter of the form \begin{equation}\label{ag_par}
    \Theta = \theta_{1}^T \oplus \theta_{2}^T \oplus \cdots \oplus \theta_{M}^T.
\end{equation}
Therefore, it is initially known that 
\begin{equation}\label{1}
    \mathcal{D}_0 :=\{ \mathbf{X} \in \mathcal{D}: \Theta\mathbf{X} \in \mathcal{Y}, \forall \Theta \in \mathcal{Q}\}
\end{equation}
is contained in $\mathcal{D}_s$.
We assume $\mathcal{D}_0$ has a nonempty interior.

\begin{assumption}
\label{A4}

The decision set $\mathcal{D}_0$ is a compact set with nonempty interior such that there exists $\mathbf{v} \in \mathbb{R}^{Md}$ and an open ball $\mathcal{B}_{2}(r)$ for which $\mathbf{v} + \mathcal{B}_{2}(r)$ is a subset of $\mathcal{D}_0$.
\end{assumption}

\section{SAFE-PRIVATE LIN-UCB}

Our algorithm for private and safe bandits is built on the Upper Confidence Bound (UCB) framework. The privacy level of each agent $m$ is an input to the algorithm, and  we will optimize it later to achieve a given regret bound while respecting the safety constraints.

We provide a high level description of the algorithm first, which proceeds in two phases. During the initial pure exploration phase, for a total number of $T'$ rounds, the central coordinator  selects actions   uniformly at random from a closed ball within the set of \textit{known safe actions} $\mathcal{D}_0$ in \eqref{1}. 
Then, it receives private versions of the observed rewards, which are subjected to Gaussian noise, and retains a record of all observed action-reward pairs. Subsequently, the Exploration-Exploitation phase is initiated. For each round $t \geq T'$, the central coordinator uses the data it has accumulated  to create a regularized least square estimate of all $\theta_m$'s, denoted by $\hat{\theta}_{t,m}$, along with a high probability confidence set $C_{t}$ to which {$\Theta_*$} belongs. These confidence sets are used to form an inner-approximation of the set of safe actions $\mathcal D_s$. The UCB rule is   employed to select an action from these inner-approximations.

\begin{algorithm}[h]
\caption{Safe-Private Lin-UCB}\label{algo}
\begin{algorithmic}[1]
\State \textbf{Inputs}: $T$, $T'$, $\mathcal{Y}$, $\mathcal{D}$, $\nu$, $S$, $K$, $R$, $\delta'$, $\delta$, $\sigma$,  $\epsilon_m$and $\epsilon$ 
\State \textbf{Pure Exploration}
    \For{\texttt{$t = 1,2,\cdots,T' $}}
    \State \textit{Central Coordinator (CC) Side:} 
        \State Randomly choose $\mathbf{X}_t \in \mathcal{D}_0 $(defined in \ref{1}) such that \eqref{random1} holds and send $\mathbf{x}_{t,m} $ to agent $m \in [M]$
        \State \textit{Agents Side:}
        \For{\texttt{$m = 1,2,\cdots,M $}}
            \State Take action $\mathbf{x}_{t,m} $ and send $ u_{t,m} = y_{t,m} + h_{t,m}$ ($h_{t,m} \sim \mathcal{N}(0,\alpha_m^2\sigma^2)$)  back to CC with $\alpha_m = \frac{\epsilon}{\epsilon_m}$. 
        \EndFor
    \EndFor

\State \textbf{Exploration - Exploitation} 
    \For{\texttt{$t = T'+1,T'+2,\cdots,T $}}
    \State \textit{Central Coordinator Side:}
        \For{\texttt{$m = 1,2,\cdots,M $}}
            \State Set $\hat{\theta}_{t,m}$(defined in\eqref{THh}).
        \EndFor
        \State Set $\mathcal{C}_{t}$ (defined in \ref{C-set-def}) and $\mathcal{D}_{t} $ (defined in \ref{safe-set-def})
            \State Choose $(\mathbf{X}_{t}, \Tilde{\Theta}_t)= \underset{ (\mathbf{X},\Theta) \in \mathcal{D}_{t} \times \mathcal{C}_{t}}{\argmax}\,f(\Theta\mathbf{X})$ and send $\mathbf{x}_{t,m}$ to agent m for the next round.
        \State \textit{Agents Side:}
            \For{\texttt{$m = 1,2,\cdots,M $}}
                \State Take action $\mathbf{x}_{t,m} $ and send $ u_{t,m} = y_{t,m} + h_{t,m}$ ($h_{t,m} \sim \mathcal{N}(0,\alpha_m^2\sigma^2)$)  back to CC with $\alpha_m = \frac{\epsilon}{\epsilon_m}$. 
            \EndFor
    \EndFor
\end{algorithmic}
\end{algorithm}

\subsection{LDP Guarantee}
Based on Assumption \ref{A2}, $y_{t,m} \in [-1,1]$. Using Lemma \ref{LDP_M}, let $\sigma = \frac{2}{\epsilon}\sqrt{2\ln(\frac{1.25}{\delta})}$ for some $\epsilon>0$ and $\delta>0$. Then let $\alpha_m \in \mathbb{R}$ be the parameter corresponding to various level of privacy for different bandits such that $\sigma_{m} = \alpha_m\sigma$ and
\begin{equation}\label{privacy_level}
    \alpha_m = \frac{\epsilon}{\epsilon_m}.
\end{equation}
For brevity, we will assume that ${\delta}$ is the same for all bandits. After playing action $\mathbf{x}_{t,m}$, each bandit adds the noise $ 
    h_{t,m} \overset{\mathrm{iid}}{\sim} \mathcal{N}(0,\alpha_m^2\sigma^2)$
to its observed response $y_{t,m}$ and sends the perturbed response $u_{t,m}$ to the coordinator. Thus, there are $M$ separate  mechanisms, each respectively protecting separate $y_{t,m}$ at time  $t$. One can then apply Lemma \ref{LDP_M} independently to each agent to obtain the following LDP result.

\begin{Fact}\label{LDP-Algo}
     At any single time $t \in [T]$, Algorithm \ref{algo} is $(\epsilon_m,\delta)$-LDP for the agents' responses $y_{t,m}, ~ m\in [M]$.
\end{Fact}

   \subsection{Pure Exploration}\label{pure_exp}
The pure exploration phase (explained in detail in Appendix \ref{Random_Act})  is run for a total of $T'$ rounds. As discussed in \cite{amani2019linear}, the proper choice of $T'$ helps to control the growth of regret by establishing a positive lower bound on $\check{\lambda}$, the minimum eigenvalue of the covariance matrix $\mathbb{E}[\mathbf{x}_{t,m}\mathbf{x}_{t,m}^T]$. Specifically, for $t > T'$, for all agents $m \in [M]$: 
\begin{equation}\label{random1}
    \check{\lambda}(\Sigma(\mathbf{x}_{t,m})) > \frac{r^2}{4Md} > 0 .
\end{equation}



 
\subsection{Exploration - Exploitation }\label{exp-exp}
Since this algorithm is a UCB based algorithm, we implement the OFU (optimism in the face of uncertainty) principle. 
For rounds $ t = T'+1 ,T'+2,\cdots, T$, the cental coordinator first obtains a regularized least-squares estimate of $\theta_{*,m}$,  named $\hat{\theta}_m$, with the regularizer $\nu > 0$ as follow: 
\begin{equation}\label{THh}
    \hat{\theta}_{t,m} =\mathbf{G}_{t,m}^{-1}\mathbf{g}_{t,m},
\end{equation}
where $\mathbf{G}_{t,m}\!=\!\sum\limits_{i=1}^{t-1} (\mathbf{x}_{i,m}\mathbf{x}_{i,m}^T)\!+\!\nu I,$ and $\mathbf{g}_{t,m}\!=\!\sum\limits_{i=1}^{t-1}\!(u_{i,m}\mathbf{x}_{i,m})$.


The central coordinator then uses $\hat{\theta}_{t,m}$ to build a confidence set $\mathcal{C}_{t}$   based on the following theorem.

\begin{theorem}\label{C-set}
 (According to Theorem 2 in \cite{abbasi2011improved}) Let Assumptions 1 and 2 hold. Then for any $\delta' \in (0,1/2)$, and for $\Theta$  defined in \eqref{ag_par}, it holds that for all $t \ge 0$, $\Theta_*$ lies in the set
 \begin{equation}\label{C-set-def}
        \mathcal{C}_{t} = \{ \Theta \in \mathbb{R}^{M \times Md} : ||\theta_m - \hat{\theta}_{t,m}||_{\mathbf{G}_{t,m}} \leq \beta_{t,m} \forall m \in [M] \}, \end{equation}
with probability of $1-\delta'$, where 
 \begin{equation}\label{beta}
     \beta_{t,m} := (\sqrt{R^2 +\alpha_m^2\sigma^2})\sqrt{d\log{( \frac{1+ tK^2/\nu}{\delta'/M})}} + S\sqrt{\nu}.
\end{equation}
\end{theorem}

Note that the Gaussian noise added to the observed response $y_{t,m}$ results in a larger lower bound on the confidence intervals $\beta_{t,m}$ compared to non-private bandits.
 
To ensure satisfaction of the safety constraints, the central coordinator then forms a   conservative safe decision set $\mathcal{D}_t$:
\begin{equation}\label{safe-set-def}
    \mathcal{D}_t :=\{ \mathbf{X} \in \mathcal{D} : \Theta\mathbf{X} \in \mathcal{Y}, \forall \Theta \in \mathcal{C}_{t}\}.
\end{equation}
Note that $\mathcal{D}_t$ is a conservative inner-approximation of the true set of safe actions $\mathcal D_s$ since the actions in $\mathcal{D}_t$ are safe for all $\Theta \in \mathcal{C}_{t}$. Given that $\Theta_* \in \mathcal{C}_{t}$ with high probability over the run of the algorithm,  actions in $\mathcal{D}_t$ are safe with same probability. 

Lastly, the central coordinator applies OFU to choose the action vector  $\mathbf{X}_{t}$ by solving $(\mathbf{X}_{t} , \Tilde{\Theta}_{t} )= \underset{ (\mathbf{X},\Theta) \in \mathcal{D}_{t} \times \mathcal{C}_{t}}{\argmax}\,f(\Theta\mathbf{X})$ and sends $\mathbf{x}_{t,m}$ back to each agent $m$ for the next round.

\section{Regret Analysis}

Our goal is to bound the regret $R_T$ in \eqref{regret}. To do so, we focus on polytopic safe decision  $
     \mathcal{Y} = \{ \mathbf{x} \in \mathbb{R}^{M}: \mathbf{A}\mathbf{x} \leq \mathbf{b}\}$,
where the matrix $\mathbf{A} = [\mathbf{a}_1,\mathbf{a}_2,\cdots ,\mathbf{a}_P]^T$ and the vector $ \mathbf{b} = [b_1,b_2,\cdots,b_P]^T$ form a set of $P$ non-redundant constraints that must be satisfied by all vectors $\mathbf{x}$ that are safe.


For analysis purposes we first transform the safe set $\mathcal{Y}$ with the positive definite, diagonal transformation matrix 
\begin{equation}\label{transform}
    \mathbf{B}  = \diag\left(\begin{bmatrix}
  \frac{1}{\beta_{T,1}}&  \frac{1}{\beta_{T,2}}& \cdots &\frac{1}{\beta_{T,M}}
\end{bmatrix}\right),
\end{equation}
which allows us to normalize the confidence interval for all the agents.
We define this transformed (scaled) safe set as:
\begin{equation*}
\begin{split}
    \mathcal{Y}' = \mathbf{B}\mathcal{Y} & = \{\mathbf{B}\mathbf{x} \in \mathbb{R}^{M}:\mathbf{A}\mathbf{x} \leq \mathbf{b}\} =  \{\mathbf{x} \in \mathbb{R}^{M}:\mathbf{A}'\mathbf{x} \leq \mathbf{b}\}, 
\end{split}
\end{equation*}
where   $\mathbf{A}' = \mathbf{A}\mathbf{B}^{-1}$.



Now, we use the notion of \emph{sharpness} of a set from \cite{pmlr-v211-hutchinson23a} to characterize how uncertainty about the bandit parameters $\Theta_*$, will affect the geometry of the conservative safe decision set $\mathcal D_t$ in \eqref{safe-set-def} and as a result affect the regret. This notation will be used to bound the distance between the optimal expected response $\Theta_* \mathbf{X}_*$, which lies in $\mathcal Y$, and the shrunk version of the safe set (Definition \ref{shrunk_set}), within which the algorithm requires the expected responses $\Theta \mathbf{X}_t$ to lie due to uncertainty.
\begin{definition}\label{shrunk_set}
    For a compact set $\mathcal{Y}\subset\mathbb{R}^M $, norm $||.||$ and a non-negative scalar $\Delta$, a \textbf{shrunk version} of  $\mathcal{Y}$ defined as $ \mathcal{Y}_{\Delta}^{||.||} := \{x \in \mathcal{Y} : x+v \in \mathcal{Y}, \forall v \in \Bar{\mathcal{B}}_{||.||}(\Delta)\}$.
\end{definition}

\begin{definition}\label{max_shrinkage}
    For a compact set $\mathcal{Y}\subset\mathbb{R}^M $ and norm$||.||$, the \textbf{maximum shrinkage} of $\mathcal{Y}$ is $ \mathcal{H}_{\mathcal{Y}}^{||.||} := \sup\{\Delta : \mathcal{Y}_{\Delta}^{||.||} \neq \emptyset \}$.
\end{definition}

\begin{definition}\label{sharpness}
    For a compact set $\mathcal{Y} \subset \mathbb{R}^M $ and norm $||.||$, the \textbf{sharpness} of $\mathcal{Y}$ is defined as $\mbox{Sharp}_{\mathcal{Y}}^{||.||} (\Delta) := \underset{x \in \mathcal{Y}}{\sup}\underset{y \in \mathcal{Y}_{\Delta}^{||.||}}{\inf} ||y - x||_2$ ,
    for all non-negative $\Delta$ such that $\mathcal{Y}_{\Delta}^{||.||} $ is nonempty.
\end{definition}
The next theorem will state our general regret results for polytopic safe sets.A detailed proof of Theorem \ref{main_2} is provided in Appendix\ref{Appen3}.

\begin{theorem}\label{main_2}
    Let all the Assumptions \ref{A1}-\ref{A4} hold. With probability at least $1-2\delta'$ the regret of Algorithm \ref{algo} is bounded~as:  
    \begin{equation}\nonumber
        \begin{split}
            R_T & \leq 2LKST'\sqrt{M}n +  L\tilde{\beta}_T(T-T')\mbox{Sharp}_{\mathcal{Y'}}^{\infty}\left(\!\frac{2\sqrt{2}K}{\sqrt{2\nu + T'\check{\lambda}}}\!\right)\\
            & +L\max({\mathcal{H}_{\mathcal{Y'}}^{\infty}}, 2) \\ &  \times \sqrt{2d\log\left(1 + \frac{TK^2}{d\nu}\right)(T-T')\sum_{m \in [M]}\beta_{T,m}^2},
        \end{split}
    \end{equation}
for any $T'\ge \max(t_{\delta'} , t'_h)$, where $t_{\delta'} := \frac{8K^2}{\check{\lambda}}log(\frac{d}{\delta'})$ , $\,t'_h:=\frac{8K^2}{\check{\lambda}(\mathcal{H}_{\mathcal{Y}'}^\infty)^2}-\frac{2\nu}{\check{\lambda}}$ and $\tilde{\beta}_T = \underset{m \in [M]}{\max\beta_{T,m} }$.
\end{theorem}

\section{The Privacy Safety Tradeoff}
To provide a better interpretation for the results in Theorem \ref{main_2} and the interplay between privacy, regret, and safety in data-driven learning, we will focus on the special case of the safe set being a simplex. Specifially, for a desired ``regret budget'', we suggest an unimprovable vector of privacy level for agents, meaning that any unilateral increase in the privacy level of one agent without decreasing another's  will not be feasible.

We define an $M$-dimensional simplex below, and for simplicity, assume that the simplex is an arbitrary diagonal transformation of the standard simplex that contains the origin.
Precisely, we consider a simplex of the form,
   \begin{equation}\label{H-simp}
        \mathit{S} = \mathit{S}(\mathbf{A},\mathbf{b}):= \{ \mathbf{x} \in \mathbb{R}^M: \mathbf{A}\mathbf{x} \leq \mathbf{b} \},
    \end{equation} 
where \(\mathbf{A}\) is an \((M+1) \times M\) matrix and \(\mathbf{b}\) is an \((M+1)\)-dimensional vector given by,
\begin{equation}\label{A_matrix}
    \mathbf{A} = \begin{bmatrix}
  \frac{1}{c_1} & \frac{1}{c_2} & \cdots & \frac{1}{c_M}\\ -1 & 0 & \cdots & 0 \\ 0 & -1 & \cdots & 0 \\ \vdots  &  & \ddots & \vdots  \\ 0 & \cdots & 0 & -1\\
    \end{bmatrix},
\end{equation}
and 
\begin{equation}\label{b}
    \mathbf{b} = \begin{bmatrix}
    \frac{1}{2} & \frac{1}{2q} & \cdots & \frac{1}{2q} \\
    \end{bmatrix}^T,
\end{equation}
with \(q = \sum_{m=1}^{M} \frac{1}{c_m}\) and \(c_m > 0\) for all \(m \in [M]\).





Let us now define the so-called privacy vector $\mathit{a}=[\alpha_1,\alpha_2,...,\alpha_M]$  as the vector of all $M$ agents' privacy level parameters (defined in \eqref{privacy_level}).
Furthermore, let $r(T, \mathit{a})$ be the regret bound in Theorem \ref{main_2} when the safe set is the simplex $\mathcal{Y}=\mathit{S}$, which is known to be $\mathcal{O}(T^{2/3}(\log (T))^{1/3})$ with an appropriate choice of algorithm parameters (as we show in Appendix\ref{bound_lim}).
Note that the constant in this regret bound $r(\mathit{a})  = \underset{T \rightarrow \infty}{\lim}\frac{r(T, \mathit{a})}{\left(T^{2}(\log T)\right)^{\frac{1}{3}}}$ is affected by the privacy levels $a$.  
Therefore, provided that we have a budget $U$ on $r(\mathit{a})$, we can consider the set of allowable privacy levels, $\mathcal{A} := \{ \mathit{a} \in \mathbb{R}^M_+ : r(\mathit{a}) \leq U \}.$
Given $\mathcal{A}$, we can then study the set of privacy levels within $\mathcal{A}$ that are unilaterally \emph{unimprovable},
\begin{equation}\label{pareto}
    \mathcal{A}_* := \{ \mathit{a}\in \mathcal{A} : (\mathit{a} + \mathbb{R}_+^M) \cap \mathcal{A} = \{ \mathit{a} \} \}.
\end{equation}
The following theorem identifies a point in $\mathcal{A}_*$. The proof of Theorem \ref{main_3} is provided in Appendix\ref{Appen4}.

\begin{theorem}\label{main_3}
    Consider the privacy vector $\mathit{a}^* \in \mathbb{R}^M$, for which the $m$th element is defined as
    \begin{equation*}
        \alpha^*_m = \sqrt{\left(\frac{R^2}{\sigma^2}  +\tilde{r}^2\right)\frac{c_m^2}{\Tilde{c}^2} - \frac{R^2}{\sigma^2}},
    \end{equation*}
    where $\tilde{c} = \max_{m \in [M]} c_m$ and
    \begin{equation*}
        \tilde{r} =\frac{1}{\sigma}\sqrt{\frac{U^3}{8 L^3 K^3 d M^{\frac{3}{2}} \left(\frac{2 S}{\hat{\lambda}} + \sqrt{4M - 3}\right)^3} - R^2}.
    \end{equation*}
    It holds that $\mathit{a}^* \in \mathcal{A}_*$.
\end{theorem}

\begin{Corollary} \label{P_S_T}
    When $R = 0$, $\alpha^*_m = \tilde{r}\frac{c_m}{\Tilde{c}}$  for all $ m \in [M]$.
\end{Corollary}
 
Corollary \ref{P_S_T}   shows that, in the case of no bandit noise $\eta_{t,m} = 0$, this unimprovable privacy vector in  Theorem \ref{main_3} reduces to an easily interpretable form. 
In particular,   the privacy level for each user should be chosen proportional to the ``tightness" of the constraint in that user's direction, i.e. $c_m$ for user $m$.
This suggests that  agents located in regions subject to tighter constraints may be afforded lower privacy levels.

\section{Numerical Experiments}\label{num_exp}
To support our theoretical results, we conducted numerical experiments in a 3-dimensional action space with $M = 3$ agents, over a time horizon of $T = 30,000$ steps. Due to space constraints, detailed discussion of the experiments is provided in Appendix\ref{app:Appen4}.  Figure \ref{norm} confirms that the privacy vector $\mathit{a}^*$ provided in Theorem \ref{main_3} results in lower regret.


\begin{figure}[htbp]
 \centering
 \includegraphics[width=0.3\textwidth]{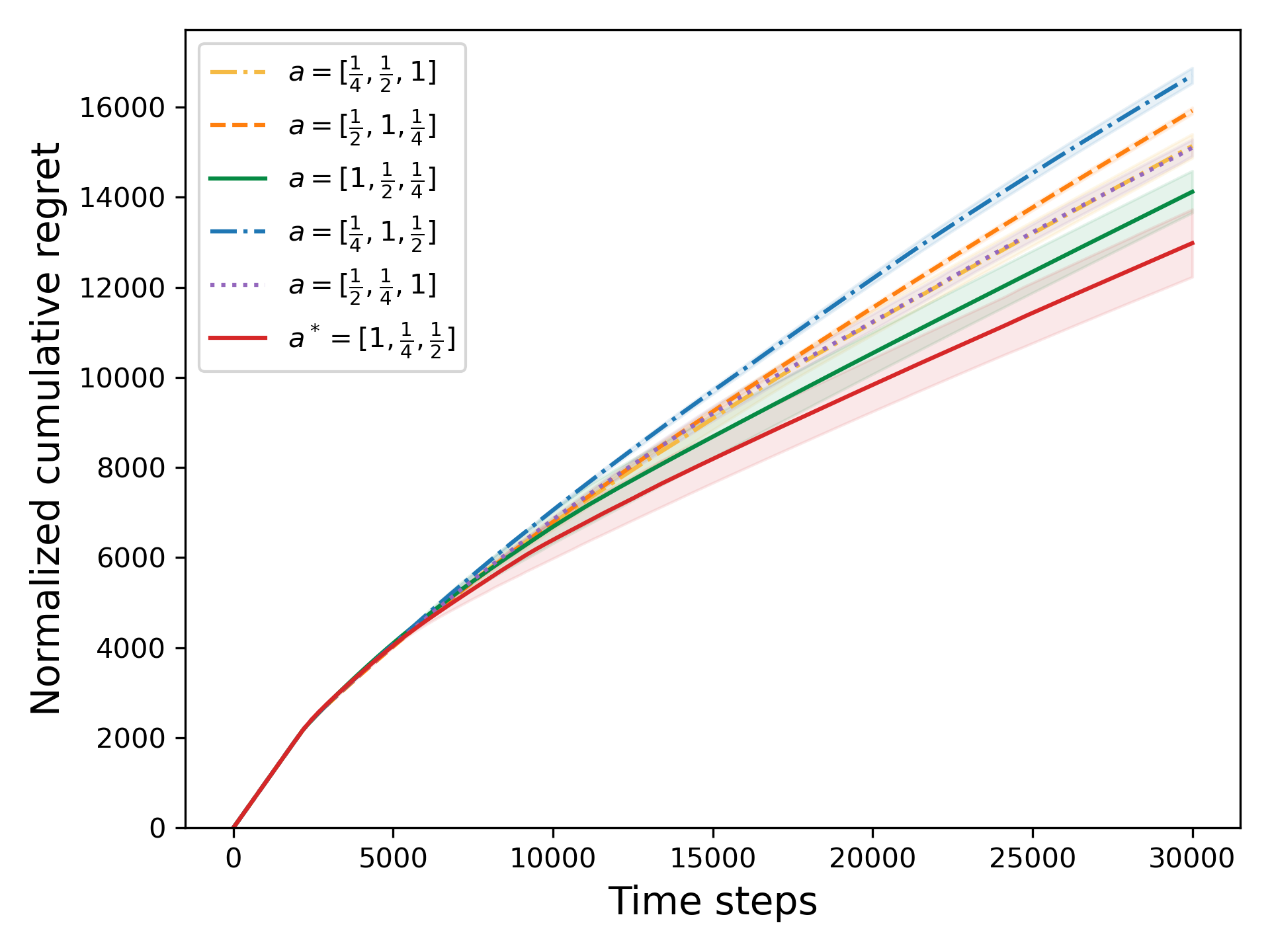}
 \caption{Average normalized regret over 3 setup for each privacy vector.} 
 \label{norm}
 \end{figure}

\section{conclusion}\label{conclusion}

We studied a system of various stochastic linear bandit problems coupled by a global safety constraint, where a non-trusted central coordinator collects agents' bandit responses protected by a local differential privacy mechanism and selects the next set of agents' actions while ensuring stage-wise safety. We suggest a regret bound that captures the impact of varying privacy levels among agents for polytopic safety sets. Then, for a fixed regret bound and a simplex constraint set, we provide a set of unilaterally unimprovable privacy levels for agents. 
It can be of interest to explore this problem across various types of constraint sets and under different privacy frameworks.

\section*{Acknowledgment}
This work was supported by NSF grant $\#2419982$.

\newpage
\bibliographystyle{IEEEtran}
\IEEEtriggeratref{13}
\bibliography{ref}










\appendices
\newpage

\section*{Appendix}
\subsection{Privacy Guarantee}\label{app:Appen1}

\begin{lemma}\label{post-processing}
    (post-processing property, proposition 2.1 \cite{dwork2014algorithmic}) If $\mathcal{A}:\mathcal{Y} \rightarrow \mathcal{U} $ is $(\epsilon_m , \delta)$-LDP and $f:\mathcal{U} \rightarrow \mathcal{G} $ is a fixed map, then $f \circ \mathcal{A} :  \mathcal{Y} \rightarrow \mathcal{G}$ is $(\epsilon_m , \delta)$-LDP
\end{lemma}
Using the post-processing property, since agent $m$ sends out the output of mechanism $\mathcal{A}_m$, Safe-Private Lin-UCB is $(\epsilon_m , \delta)$-LDP protecting $y_{t,m}$.

\subsection{Proof of Theorem \ref{C-set}}\label{Appen2}

We first show that $\Theta_*$ is in the confidence set with high probability. 
The proof is based on the method in \cite{abbasi2011improved} used for linear UCB algorithms. 

Given that at each round $t$, each agent sends the perturbed version of its response $u_{t,m} = y_{t,m} + h_{t,m}$ and given that  $y_{t,m} = \theta_{*,m}^{T}\mathbf{x}_{t,m} + \eta_{t,m}$, it holds that $u_{t,m} = \theta_{*,m}^{T}\mathbf{x}_{t,m} + (\eta_{t,m} + h_{t,m})$. Then let us define $w_{t,m} := \eta_{t,m} + h_{t,m} $.

\begin{lemma}\label{martingle}
For all $t \in [T]  $ and $ m \in [M]$ , $w_{t,m}$ is conditionally zero-mean $R'$-sub-Gaussian for a fixed constant $R' = \sqrt{R^2 + \sigma_{m}^2}$, i.e. 
\begin{equation}
\begin{split}
     & \mathbb{E}[w_{t,m}|\mathcal{F}_{t-1}] = 0 \\ 
     & \mathbb{E}[e^{\lambda w_{t,m}}|\mathcal{F}_{t-1}] \leq exp(\lambda^2\mathit{R'}^2/2),  \,\,\, \forall \lambda \in \mathbb{R} \\
\end{split}
\end{equation}
where $\mathcal{F}_{t-1} = \sigma(\mathbf{x}_{1:t},\eta_{1:t-1 , 1:M},h_{1:t-1 , 1:M}).$
\end{lemma}

\begin{proof}
Based on Assumption \ref{A1}, for all $m \in [M]$ and $ \{\eta_{m}\}_{t=1}^{\infty}$, the $\eta_{t,m}$ is conditionally zero-mean R-sub Gaussian. Since $h_{t,m} \overset{\mathrm{iid}}{\sim} \mathcal{N}(0,\alpha_m^2\sigma^2)$ is also zero-mean Gaussian independent random noise with variance $\sigma_{m}^2 = \alpha_m^2\sigma^2$, it holds that:
      \begin{flalign*}
        \mathbb{E}[w_{t,m}|\mathcal{F}_{t-1}] &= \mathbb{E}[\eta_{t,m} + h_{t,m}|\mathcal{F}_{t-1}] \\
        & = \mathbb{E}[\eta_{t,m}|\mathcal{F}_{t-1}] + \mathbb{E}[h_{t,m}|\mathcal{F}_{t-1}] \\
        & = 0 + \mathbb{E}[h_{t,m}|\mathcal{F}_{t-1}] \\
        & = 0 + \mathbb{E}[h_{t,m}] \\
        & = 0,\\
      \end{flalign*}
      
    and 
    
       \begin{flalign*}
        \mathbb{E}[e^{\lambda w_{t,m}}|\mathcal{F}_{t-1}] &= \mathbb{E}[e^{\lambda (\eta_{t,m} + h_{t,m})}|\mathcal{F}_{t-1}] \\
        & = \mathbb{E}[e^{\lambda (\eta_{t,m})}|\mathcal{F}_{t-1}]\mathbb{E}[e^{\lambda (h_{t,m})}|\mathcal{F}_{t-1}]\\
        & \leq  exp(\lambda^2\mathit{R}^2/2)exp(\lambda^2\sigma_{m}^2/2)\\
        & = exp(\lambda^2(\mathit{R}^2 + \sigma_{m}^2)/2) \\
         & = exp(\lambda^2\mathit{R'}^2/2).\\
      \end{flalign*}
\end{proof}
Adopting the Theorem 2 of \cite{abbasi2011improved}, with a probability of $1 - \alpha' $, where $\alpha' \in (0,1/M)$, for all rounds $t > T'$
\begin{equation*}
   ||\theta_m - \hat{\theta}_{t,m}||_{\mathbf{G}_{t,m}} \leq R'\sqrt{d\log{\left( \frac{1+ tK^2/\nu}{\alpha'}\right)}}.
\end{equation*}
Since this bound holds for each unknown bandit parameter of each agent independently, we can say, for $\delta' = M\alpha'$, Theorem \ref{C-set} holds. 

\subsection{Proof of Theorem \ref{main_2}}\label{Appen3}

Given the regret of  Safe-Private Lin-UCB, $R_T$ in \eqref{regret}, denote the regret accumulated at time $t$ as $r_t = f(\Theta_{*}\mathbf{X}_*) - f(\Theta_{*}\mathbf{X}_t)$. This  can be decomposed in the following form:
\begin{equation}\label{regret_term}
\begin{split}    
    r_t  & =  \underbrace{f(\Theta_{*}\mathbf{X}_*)  -f(\Tilde{\Theta}_t\mathbf{X}_t)}_{\mbox{Term I}}+ \underbrace{ f(\Tilde{\Theta}_t\mathbf{X}_t)  -  f(\Theta_{*}\mathbf{X}_t)}_{\mbox{Term II}}. \\
\end{split}
\end{equation}

To find the upper bound on the regret  we break the proof into three main parts:
\begin{itemize}
    \item The regret of safety (Term I) during the exploration-exploitation phase;
    \item The regret of Term II during the exploration-exploitation phase;
    \item The regret of the pure exploration phase.
\end{itemize}

\subsubsection{Random Action Selection for Pure Exploration Phase}\label{Random_Act}
We start by bounding Term I. To do so, we first explant the process of action selection during the pure exploration phase. During this phase the central coordinator chooses actions for each user uniformly at random from a subset of the set of known safe actions $\mathcal{D}_0$ in \eqref{1}. The goal of this pure exploration phase is to provide a lower bound  on the minimum eigenvalue of the covariance matrix $\mathbb{E}[\mathbf{x}_{t,m}\mathbf{x}_{t,m}^T]$  for $t > T'$.
Based on Assumption \ref{A4}, $\mathcal{D}_0$ has a nonempty interior and as such, there exists $\mathbf{v} \in \mathbb{R}^{Md}$ such that the open ball $\mathbf{v} + \mathcal{B}_{2}(r)$ is a subset of $\mathcal{D}_0$. As a result, $\mathbf{v} + \Bar{\mathcal{B}}_{2}(\frac{r}{2})$ is a subset of $\mathcal{D}_0$ as well. To sample from this set, define a vector $\mathbf{u}$ uniformly sampled from a unit sphere i.i.d, with, $\mathbb{E}[\mathbf{u}\mathbf{u}^T] = \frac{1}{Md}I$. The central coordinator  chooses $\mathbf{X}_t$ as follow:
\begin{equation}\label{random} 
\mathbf{X}_t = \mathbf{v} + \frac{r}{2}\mathbf{u}.
\end{equation}
Since $\mathbf{X}_t = [\mathbf{x}_{t,1}^T , \mathbf{x}_{t,2}^T , \cdots , \mathbf{x}_{t,M}^T ]^T$ is a stacked vector of all agents' actions at round $t$, for $\mathbf{v}= [\mathbf{v}_{1}^T , \mathbf{v}_{2}^T , \cdots , \mathbf{v}_{M}^T ]^T$ and given $\mathbb{E}[\mathbf{X}_t\mathbf{X}_t^T] = \mathbf{v}\mathbf{v}^T + \frac{r^2}{4Md}I$, we can state that for each agent $m$, we have $\mathbb{E}[\mathbf{x}_{t,m}\mathbf{x}_{t,m}^T] = \mathbf{v}_m\mathbf{v}_m^T + \frac{r^2}{4Md}I$  and $\check{\lambda}(\Sigma(\mathbf{x}_{t,m})) > \frac{r^2}{4Md} > 0 $, where $\check{\lambda}$ denotes the minimum eigenvalue.
 
Now we adopt the following Lemma from \cite{amani2019linear} to bound the minimum eigenvalue of the Gram matrix 
\begin{equation*}
    \mathbf{G}_{T',m} =\sum_{i=1}^{T'-1} (\mathbf{x}_{i,m}\mathbf{x}_{i,m}^T)+\nu I,
\end{equation*} for all agents, which is needed further to define the shrunk version of the safe set. 
\begin{lemma}\label{eigen}
(Lemma 1 in \cite{amani2019linear}) With probability at least $1-\delta'$, where $\delta' \in (0,1/2)$, it holds: 
    \begin{equation}
        \check{\lambda}(\mathbf{G}_{T',m}) \ge \nu + \frac{T'\check{\lambda}}{2},
    \end{equation}
where $T' \geq \frac{8K^2}{\check{\lambda}}log(\frac{d}{\delta'}) =: t_{\delta'}$ for all $m \in [M]$.
\end{lemma}

For Lemma \ref{eigen} and Theorem \ref{C-set} to hold jointly with probability at least $1 - \delta$ by the union bound, we need each of them to hold with probability at least $1 - \frac{\delta}{2}$. Therefore,  for $\delta' = \frac{\delta}{2}$, when $\delta' \in (0,1/2)$, they jointly hold with probability $1 - \delta$.
For the remainder of the proof, we will condition on this holding without further reference.

Now, we bound Term I directly in the following.
\begin{lemma}\label{Term1}
     Let Assumptions \ref{A1}-\ref{A4} hold. Then for $t > T'$, Term I of the regret at each round bounded as 
    \begin{equation}
        f(\Theta_{*}\mathbf{X}_*)  -f(\Tilde{\Theta}_t\mathbf{X}_t) \leq L\tilde{\beta}_T\mbox{Sharp}_{\mathcal{Y}'}^{\infty}(\frac{2\sqrt{2}K}{\sqrt{2\nu + T'\check{\lambda}}}),
    \end{equation}
    where $T' \ge \max(t_{\delta'} , t'_h)$ with some probability. 
\end{lemma}

\begin{proof}
To prove Lemma \ref{Term1} we define the expanded confidence set  
\begin{equation}\label{Expand_con_set}
    \Tilde{\mathcal{C}}_{t} = \{ \Theta \in \mathbb{R}^{M \times Md} : ||\theta_m - \theta_{*,m}||_{\mathbf{G}_{t,m}} \leq 2\beta_{t,m}, \forall m  \},
\end{equation}
for all agents, where $ \mathcal{C}_{t} \subseteq \Tilde{\mathcal{C}}_{t}$.
This follows the fact that 
\begin{equation}
    \begin{split}
        ||\theta_m - \theta_{*,m}||_{\mathbf{G}_{t,m}}  = & ||\theta_m - \hat{\theta}_{t,m} + \hat{\theta}_{t,m} -\theta_{*,m}||_{\mathbf{G}_{t,m}} \\
         \leq & ||\theta_m - \hat{\theta}_{t,m}||_{\mathbf{G}_{t,m}} + \\ & ||\hat{\theta}_{t,m} -\theta_{*,m}||_{\mathbf{G}_{t,m}} \\
         \leq & 2\beta_{t,m}.
    \end{split}
\end{equation}

Then, in order to prove the statement of the lemma, we will define a set that is contained in the (transformed) set of feasible responses from the central coordinator,
\begin{equation}
    \Tilde{\mathcal{Y}}_t' := \{\mathbf{B} \Theta\mathbf{X}: \mathbf{X} \in \mathcal{D}_t \,, \Theta \in \mathcal{C}_t\},
\end{equation}
and will therefore allow us to bound the distance between the optimal response $\Theta_{*}\mathbf{X}_*$ and the optimistic response $\Tilde{\Theta}_t\mathbf{X}_t$.
In particular, we will show that, with an appropriate choice of shrinkage $\ell'$, the shrunk version (defined in Definition \ref{shrunk_set}) of the feasible set $\mathcal{Y}'$ will be contained in $\Tilde{\mathcal{Y}}_t'$.
To appropriately choose $\ell'$, we first study the range of values that $\frac{1}{\beta_{t,m}} \theta_m^T x_m$ can take for $\Theta \in \mathcal{C}_t$ and $t > T'$, as follows:
\begin{equation}\label{Ball_l_p}
    \begin{split}
        & \frac{1}{\beta_{t,m}}\theta_{m}^T\mathbf{x}_m\\
        & \in [\frac{\theta_{*}^T\mathbf{x}_m}{\beta_{t,m}}- 2||\mathbf{x}_m||_{\mathbf{G}_{t,m}^{-1}}, \frac{\theta_{*}^T\mathbf{x}_m}{\beta_{t,m}} + 2||\mathbf{x}_m||_{\mathbf{G}_{t,m}^{-1}}]\\
        &\subseteq[\frac{\theta_{*}^T\mathbf{x}_m}{\beta_{t,m}}- \frac{2K}{\sqrt{\check{\lambda}(\mathbf{G}_{t,m})}}, \frac{\theta_{*}^T\mathbf{x}_m}{\beta_{t,m}}+ \frac{2K}{\sqrt{\check{\lambda}(\mathbf{G}_{t,m})}}] \\
        & \subseteq [\frac{\theta_{*}^T\mathbf{x}_m}{\beta_{t,m}}- \frac{2\sqrt{2}K}{\sqrt{2\nu + T'\check{\lambda}}}, \frac{\theta_{*}^T\mathbf{x}_m}{\beta_{t,m}}+ \frac{2\sqrt{2}K}{\sqrt{2\nu + T'\check{\lambda}}}] \\
        & = [\frac{\theta_{*}^T\mathbf{x}_m}{\beta_{t,m}} - \mathit{l'}, \frac{\theta_{*}^T\mathbf{x}_m}{\beta_{t,m}}+ \mathit{l'}],
    \end{split}
\end{equation}
when $\mathit{l'} = \frac{2\sqrt{2}K}{\sqrt{2\nu + T'\check{\lambda}}}$. 
Then, we can use this to find a set that is contained in $\mathcal{D}_t$ when $t > T'$,
\begin{equation}
\label{eqn:dt}
    \begin{split}
        \mathcal{D}_t &= \{ \mathbf{X} \in \mathcal{D} : \Theta\mathbf{X} \in \mathcal{Y}, \forall \Theta \in \mathcal{C}_{t}\}\\
        & = \{ \mathbf{X} \in \mathcal{D} : \mathbf{B} \Theta\mathbf{X} \in \mathcal{Y}', \forall \Theta \in \mathcal{C}_{t}\}\\
        & \supseteq \{ \mathbf{X} \in \mathcal{D} : \mathbf{B} \Theta_* \mathbf{X} + \mathbf{v} \in \mathcal{Y}', \forall \mathbf{v} \in \Bar{\mathcal{B}}_{\infty}(\mathit{l'}) \},\\
    \end{split}
\end{equation}
where the first equality is from the definition of the $\mathcal{Y}'$ and the inclusion is due to \eqref{Ball_l_p}.
Then, for $t > T'$, we have that,
\begin{equation}
\begin{split}
    \Tilde{\mathcal{Y}}_t' &= \{\mathbf{B} \Theta\mathbf{X}: \mathbf{X} \in \mathcal{D}_t \,, \Theta \in \mathcal{C}_t\}\\
    & \supseteq \{\mathbf{B} \Theta_* \mathbf{X}: \mathbf{X} \in \mathcal{D}_t\}\\
    & \supseteq \mathbf{B} \Theta_* \mathcal{D} \cap \{y: y + \mathbf{v} \in \mathcal{Y}', \forall \mathbf{v} \in \Bar{\mathcal{B}}_{\infty}(\mathit{l'}) \}\\
    & \supseteq {(\mathcal{Y}')}_{\mathit{l}'}^{\infty},
\end{split}
\end{equation}
where the first inclusion is due to the fact that $\Theta_* \in \mathcal{C}_t$, the second inclusion is due to \eqref{eqn:dt}, and the third inclusion follows from the definition of the shrunk set. 
Then, to ensure ${(\mathcal{Y}')}_{\mathit{l}'}^{\infty}$ is nonempty, we need to ensure that the time horizon $T'$ of the safe exploration phase is long enough such that $\mathit{l}' = \frac{2\sqrt{2}K}{\sqrt{2\nu + T'\check{\lambda}}} \leq \mathcal{H}_{\mathcal{Y}'}^{\infty}$.
In particular, we need that
\begin{equation}\label{th}
    T' \geq \frac{8K^2}{\check{\lambda}(\mathcal{H}_{\mathcal{Y}'}^{\infty})^2} -\frac{2\nu}{\check{\lambda}} =: t'_h.
\end{equation}
Then, combining the result of \eqref{th} with Lemma \ref{eigen}, we need $T' \geq \max(t'_h,t_{\delta'})$.

With this established, we are ready to bound Term I directly.
To do so, we define $y_* =\Theta_{*}\mathbf{X}_* $ and $\Tilde{y}_t:=\Tilde{\Theta}_t\mathbf{X}_t $ where $(\mathbf{X}_{t}, \Tilde{\Theta}_t)= \underset{ (\mathbf{X},\Theta) \in \mathcal{D}_{t} \times \mathcal{C}_{t}}{\argmax}\,f(\Theta\mathbf{X})$.
Additionally, we define $y_*' = \mathbf{B} y_*$ and $\Bar{y} \in \underset{y \in {\mathcal{Y'}}_{\mathit{l'}}^{\infty}}{\argmin} ||y_*' - y||_2$. Note that $f(\Tilde{y}_t) \geq f(\mathbf{B}^{-1} \Bar{y})$.
Then, under the conditions of the lemma,
\begin{equation}\label{General_Term1}
    \begin{split}
        \textbf{Term I} & := f(\Theta_{*}\mathbf{X}_*)  -f(\Tilde{\Theta}_t\mathbf{X}_t) \\
        & = f(y_*)-f(\Tilde{y}_t) \\ 
        & \leq f(y_*) - f(\mathbf{B}^{-1} \Bar{y}) \\
        & \leq  |f(y_*)-f(\mathbf{B}^{-1} \Bar{y})| \\
        & \leq L||y_* - \mathbf{B}^{-1} \Bar{y}||  \\
        & = L||\mathbf{B}^{-1}(y'_* - \Bar{y}')|| \\
        & \leq L||\mathbf{B}^{-1}|| ||y'_* - \Bar{y}'|| \\
        & = L \tilde{\beta}_T ||y'_* - \Bar{y}'|| \\
        & \leq L\tilde{\beta}_T \mbox{Sharp}_{{\mathcal{Y'}}}^{\infty} (\mathit{l}') \\ 
        & = L\tilde{\beta}_T\mbox{Sharp}_{\mathcal{Y}'}^{\infty}\left(\frac{2\sqrt{2}K}{\sqrt{2\nu + T'\check{\lambda}}} \right),
    \end{split}
\end{equation}
where we use the fact that the largest eigenvalue of $\mathbf{B}^{-1}$ is $\tilde{\beta}_T$.
\end{proof}

The following lemma gives an upper bound on the Term II.

\begin{lemma}\label{Term2}
    Let Assumptions \ref{A1}-\ref{A4} hold. Then it holds that
    \begin{equation}
    \begin{split}
        & \sum_{t = T'+1}^{T}\text{Term II} \\ & \leq  L\max({\mathcal{H}_{\mathcal{Y}'}^{\infty}}, 2) \sqrt{2d\log(1 + \frac{TK^2}{d\nu})(T-T')\Bigl(\!\!\sum_{m \in [M]}\!\!\!\beta_{T,m}^2\Bigl)}, 
    \end{split}
         \end{equation}
    when $T' \ge \max(t_{\delta'}, t'_h)$ with some probability.
\end{lemma}

\begin{proof}
Throughout the proof, we take $t > T'$ without further reference.
First, note that Term II can be bounded as follows:
\begin{equation}
    \begin{split}
        \textbf{Term II} & := f(\Tilde{\Theta}_t\mathbf{X}_t)- f(\Theta_{*}\mathbf{X}_t) \\
        & \leq L||\Tilde{\Theta}_t\mathbf{X}_t - \Theta_{*}\mathbf{X}_t||_2  \\
        & = L\sqrt{\sum_{m \in [M]} (\mathbf{x}_{t,m}^T\Tilde{\theta}_{t,m} - \mathbf{x}_{t,m}^T \theta_{*,m})^2} \\
        & \leq L\sqrt{\sum_{m \in [M]} r_{t,m}^2}, \\
    \end{split}
\end{equation}
where $r_{t,m} :=  \mathbf{x}_{t,m}^T\Tilde{\theta}_{t,m} - \mathbf{x}_{t,m}^T \theta_{*,m}$.
Then, using the definition of the confidence set in Theorem \ref{C-set}, we have that
\begin{equation}\label{r2_1}
\begin{split}
     r_{t,m} & :=  \mathbf{x}_{t,m}^T\Tilde{\theta}_{t,m} - \mathbf{x}_{t,m}^T \theta_{*,m} \\ 
     & \leq ||\mathbf{x}_{t,m}||_{\mathbf{G}_{t,m}^{-1}}||\Tilde{\theta}_{t,m}-\theta_{*,m}||_{\mathbf{G}_{t,m}} \\
     & \leq ||\mathbf{x}_{t,m}||_{\mathbf{G}_{t,m}^{-1}}||\Tilde{\theta}_{t,m} -\hat{\theta}_{t , m} +\hat{\theta}_{t , m} - \theta_{*,m}||_{\mathbf{G}_{t,m}} \\
     & \leq 2\beta_{t,m}||\mathbf{x}_{t,m}||_{\mathbf{G}_{t,m}^{-1}}\\
     & \leq 2\beta_{T,m}||\mathbf{x}_{t,m}||_{\mathbf{G}_{t,m}^{-1}},
\end{split}
\end{equation}
where the last inequality uses the fact that $\beta_{t,m}$ is monotone in $t$.
In order to bound $r_{t,m}$, we first note that, using the specification of the pure exploration phase, it holds that:
\begin{equation}\label{r2_2}
\begin{split}
      r_{t,m} & \leq 2\beta_{T,m}||\mathbf{x}_{t,m}||_{\mathbf{G}_{t,m}^{-1}}\\
     & \leq \frac{2\sqrt{2} \beta_{T,m} K}{\sqrt{2\nu + T'\check{\lambda}}}\\
    & \leq \frac{2\sqrt{2}\hat{\beta_T}K}{\sqrt{2\nu + T'\check{\lambda}}}\\
     & \leq \beta_{T,m} \mathcal{H}_{\mathcal{Y}}^{\infty}.
\end{split}
\end{equation}
It follows that:  
\begin{equation}
\begin{split}
       r_{t,m} & \leq \min(\beta_{T,m}\mathcal{H}_{\mathcal{Y'}}^{\infty} ,2\beta_{T,m}||\mathbf{x}_{t,m}||_{\mathbf{G}_{t,m}^{-1}} ) \\
       & \leq \beta_{T,m}\min(\mathcal{H}_{\mathcal{Y'}}^{\infty}, 2||\mathbf{x}_{t,m}||_{\mathbf{G}_{t,m}^{-1}})\\
       & \leq \beta_{T,m}\max(\mathcal{H}_{\mathcal{Y'}}^{\infty}, 2) \min(1,||\mathbf{x}_{t,m}||_{\mathbf{G}_{t,m}^{-1}})\\
       & \leq \beta_{T,m}\max(\mathcal{H}_{\mathcal{Y'}}^{\infty}, 2) \min(1,||\mathbf{x}_{t,m}||_{\mathbf{G}_{t,m}^{-1}}).\\
\end{split}
\end{equation}

To continue, we use the so-called elliptic potential lemma.
We give the version of this lemma from \cite{abbasi2011improved} below.
\begin{lemma}[Lemma 11 of \cite{abbasi2011improved}] 
\label{lem:ellip_pot}
Consider a sequence $\{\mathbf{x}_{t=1}^{\infty}\}$, in $ \mathbb{R}^d$ such that $||\mathbf{x}_{t}||_2 \leq K$ for all $t \in [T]$.
Also, let $\mathbf{G}_{t} = \sum_{i=1}^{t-1} \mathbf{x}_{i}\mathbf{x}_{i}^T + \nu I$, where $\nu$ is a positive scalar.
Then, it follows that,
    \begin{equation}
    \begin{split}
         \sum_{t=1}^T &\min(||\mathbf{x}_{t}||_{\mathbf{G}_{t}^{-1}}^2 , 1)\\
          &\leq 2\left(d\log(\frac{trace(\nu I)+TK^2}{d}) - \log (\det(\nu I)\right).\\
    \end{split}
    \end{equation}
\end{lemma}

It follows from Lemma \ref{lem:ellip_pot} that:
\begin{equation}
\begin{split}
     r'_{t,m} & := \sum_{t=T'}^T \min(||\mathbf{x}_{t,m}||^2_{\mathbf{G}_{t,m}^{-1}} , 1)\\
     & \leq \sum_{t=1}^T \min(||\mathbf{x}_{t,m}||^2_{\mathbf{G}_{t,m}^{-1}} , 1)\\
     &\leq 2\left(d\log(\frac{trace(\nu I)+TK^2}{d})- \log (\det(\nu I)\right)\\
      &\leq 2d\log(1 + \frac{TK^2}{d\nu}).\\
\end{split}
\end{equation}

Then, by Cauchy-Schwarz, we have that:
\begin{equation}\label{term2_bound}
    \begin{split}
        & \sum_{t=T'}^T\mathbf{Term\ II}\\
         & \leq L \sum_{t=T'}^T\sqrt{\sum_{m \in [M]} r_{t,m}^2} \\
         & \leq L \sqrt{(T-T')\sum_{m \in [M]}\sum_{t=T'}^T r_{t,m}^2} \\
         & \leq L\max(\mathcal{H}_{\mathcal{Y'}}^{\infty}, 2)\sqrt{(T-T')r'_{t,m}\sum_{m \in [M]}\beta_{T,m}^2} \\
         & \leq L\max(\mathcal{H}_{\mathcal{Y'}}^{\infty}, 2)\\  & \times 
             \left(\sqrt{2d\log\left(1 + \frac{TK^2}{d\nu}\right)(T-T')\Bigl(\!\!\sum_{m \in [M]}\!\!\!\beta_{T,m}^2\Bigl)}\right).\\
    \end{split}
\end{equation}

\end{proof}

Now we complete the proof of Theorem \ref{main_2}.
First note that, the instantaneous regret in the pure exploration phase can be bounded as
\begin{equation}\label{pure_bound}
    \begin{split}
        r_{t < T'} & := f(\Theta_{*}\mathbf{X}_*) - f(\Theta_{*}\mathbf{X}_t) \\
         & \leq L||\Theta_{*}\mathbf{X}_* - \Theta_{*}\mathbf{X}_t||_2\\
         & \leq L\sqrt{M} \underset{m\in[M]}{max}(\mathbf{x}_{*,m}^T\theta_{*,m} - \mathbf{x}_{t,m}^T \theta_{*,m})\\
         & \leq 2LKS\sqrt{M}.
    \end{split}
\end{equation}

Then, using Lemma \ref{Term1}, Lemma \ref{Term2} and \ref{pure_bound}, Theorem \ref{main_2} holds:
\begin{equation}
    \begin{split}
        R_T  = & \sum_{t=1}^{T'}r_{t < T'}+\sum_{t=T'}^T\textbf{Term I}  +\sum_{t=T'}^T\textbf{Term II} \\
         \leq &  \mbox{ } 2LKST'\sqrt{M} \\ 
         & +  L\hat{\beta_T}(T-T')\mbox{Sharp}_{\mathcal{Y'}}^{\infty}(\frac{2\sqrt{2}K}{\sqrt{2\nu + T'\check{\lambda}}})\\
        & + L\max({\mathcal{H}_{\mathcal{Y'}}^{\infty}}, 2)\\
        & \quad \times \sqrt{2d\log\left(1 + \frac{TK^2}{d\nu}\right)(T-T')\Bigl(\!\!\sum_{m \in [M]}\!\!\!\beta_{T,m}^2\Bigl)}.
    \end{split}
\end{equation}

\subsection{Proof of Theorem \ref{main_3}}\label{Appen4}

In order to prove Theorem \ref{main_3}, we will first characterize the sharpness of $\mathit{S}'$, the transformed of the simplex $\mathit{S}$ (defined in \eqref{H-simp}), in Lemmas \ref{transformed_sharp} and \ref{min_sharp}, which are proved in Appendix\ref{trans_proof} and Appendix\ref{proof_min}, and then will complete the proof in Appendix\ref{compl_proof}. First let $\rho_m = \frac{c_m}{\beta_{T,m}}$ for all $m \in [M]$.

\begin{lemma}
\label{transformed_sharp}
    For simplex $\mathit{S}'$ and for an arbitrary $0 \leq \Delta \leq \mathcal{H}_{\mathit{S}}^{\infty}$, it holds:
\begin{equation}\label{eq_sharp}
    \mbox{Sharp}_{\mathit{S}'}^{\infty} ({\Delta}) = \,\Delta \sqrt{(M-1)+ ( 2q'\tilde{\rho} - 1)^2},
\end{equation}    
where $\tilde{\rho} = \underset{m \in [M]}{\max}\rho_m$. 
\end{lemma}

\subsubsection{Proof of Lemma \ref{transformed_sharp}}
\label{trans_proof}

First, we will give an exact form of the shrunk version of the transformed simplex $\mathit{S}' = \mathit{S}(\mathbf{A}',\mathbf{b})$, where 
\begin{equation}\label{Ap_matrix}
  \mathbf{A}' = \begin{bmatrix}
  \frac{\beta_{T,1}}{c_1} & \frac{\beta_{T,2}}{c_2} & \cdots & \frac{\beta_{T,M}}{c_M}\\ -\beta_{T,1} & 0 & \cdots & 0 \\ 0 & -\beta_{T,2} & \cdots & \vdots \\ \vdots  &  & \ddots & 0  \\ 0 & \cdots & 0 & -\beta_{T,M}
\end{bmatrix},
\end{equation}
given  Equation \eqref{A_matrix} and the fact that $\mathbf{A}' = \mathbf{A}\mathbf{B}^{-1}$.

Throughout this section, we will use the notation  $q = \sum_{m=1}^M \frac{1}{c_m}$ and $ q' = \sum_{m=1}^{M} \frac{1}{\rho_m}$.
 
\begin{lemma}\label{max_shrink}
For $0 < \Delta \leq \mathcal{H}_{\mathit{S}'}^{\infty}$ the shrunk version of simplex $\mathit{S}'$ is
\begin{equation}
    \begin{split}
    (\mathit{S}')_{\Delta}^{\infty} & = \bigg\{\mathbf{x}\in\mathbb{R}^M : x_m \geq \Delta - \frac{1}{2q \beta_{T,m}} \ \forall m \in [M],\\
    & \qquad \qquad \qquad \sum_{m=1}^{M}\frac{\beta_{T,m}}{c_m} x_m \leq \frac{1}{2} - \Delta q' \bigg\}.
    \end{split}
\end{equation}
\end{lemma}

\begin{proof}
Since the simplex is defined as the intersection of $M + 1$ halfspaces, we study the sharpness of the simplex by studying the sharpness of each halfspace individually.
In particular, we write the shrunk version of each halfspace as $(\mathit{S}')_{\Delta,j}^{\infty}$ such that the shrunk version of the simplex can be defined as $(\mathit{S}')_{\Delta}^{\infty} = \bigcap_{j \in \{ 0,...,M\}} (\mathit{S}')_{\Delta,j}^{\infty}$.
With the $\mathbf{a}_j'$ denoting the $j$th row of $\mathbf{A}'$ (such that $\mathbf{A}' = [(\mathbf{a}_1')^T, \ (\mathbf{a}_2')^T, \ ...\,,\ (\mathbf{a}_{M+1}')^T ]^T$, it holds that:
\begin{equation}
\begin{split}
     (\mathit{S}')_{\Delta,j}^{\infty} & := \{ \mathbf{x}\in\mathbb{R}^M : {\mathbf{a}'_{j}}^T(\mathbf{x+v}) \leq b_j , v\in \Bar{\mathcal{B}}_{\infty}(\Delta) \} \\
     & = \{ \mathbf{x}\in\mathbb{R}^M : {\mathbf{a}'_{j}}^T\mathbf{x}+ \max_{v\in \Bar{\mathcal{B}}_{\infty}(\Delta)} \mathbf{v}^T \mathbf{a}'_{j} \leq b_j \} \\
    & = \{\mathbf{x}\in\mathbb{R}^M : {\mathbf{a}'_{j}}^T\mathbf{x} \leq b_j - 
    \Delta||\mathbf{a}'_{j}||_1 \}.
\end{split}
\end{equation}
Note that the shrunk version of each halfspace is still a halfspace.
We will then express each of these halfspaces in closed form.

First, we consider the case where $j \neq 0$.
In particular, for $j \neq 0$, it holds that ${\mathbf{a}'}_j^T\mathbf{x} = -\beta_{T,j}x_j$, $||\mathbf{a}'_{j}||_1 = \beta_{T,j}$, and $b_j = 1/2q$.
Therefore, for $j \neq 0$,
\begin{equation}\label{d_const_1}
\begin{split}
(\mathit{S}')_{\Delta,j}^{\infty} & = \left\{\mathbf{x}\in\mathbb{R}^M : -\beta_{T,m}x_m \leq \frac{1}{2 q} - \Delta \beta_{T,m} \right\}\\
& = \left\{\mathbf{x}\in\mathbb{R}^M : x_m \geq \Delta - \frac{1}{2q \beta_{T,m}} \right\}.
\end{split} 
\end{equation}
Then, we consider the case where $j = 0$.
Since ${\mathbf{a}'}_0^T\mathbf{x} = \sum_{m=1}^{M}\frac{\beta_{T,m}}{c_m} x_m$, and $||\mathbf{a}'_{0}||_1 = \sum_{m=1}^{M}\frac{\beta_{T,m}}{c_m} = q'$, and $b_0 = 1/2$, it follows that
\begin{equation}\label{d_const_2}
(\mathit{S}')_{\Delta,0}^{\infty} = \left\{\mathbf{x}\in\mathbb{R}^M : \sum_{m=1}^{M}\frac{\beta_{T,m}}{c_m} x_m \leq \frac{1}{2} - \Delta q' \right\}.
\end{equation}
Finally, combining \eqref{d_const_1} and \eqref{d_const_2}, it holds that
\begin{equation}
\begin{split}
    (\mathit{S}')_{\Delta}^{\infty} & = \bigcap_{j \in \{ 0,...,M\}} (\mathit{S}')_{\Delta,j}^{\infty}\\
    & = \bigg\{\mathbf{x}\in\mathbb{R}^M : x_m \geq \Delta - \frac{1}{2q \beta_{T,m}} \ \forall m \in [M],\\
    & \qquad \qquad \qquad \sum_{m=1}^{M}\frac{\beta_{T,m}}{c_m} x_m \leq \frac{1}{2} - \Delta q' \bigg\},
\end{split}
\end{equation}
which is completing the proof. 
\end{proof}

Next, we provide an exact expression for the maximum shrinkage of the simplex.

\begin{lemma}
      The maximum shrinkage of simplex $ \mathit{S}'$ is $\mathcal{H}_{\mathit{S}'}^{\infty} = \frac{1}{2q'}$.
\end{lemma}

\begin{proof}
Since $\mathcal{H}_{\mathit{S}}^{\infty} = \sup\{\Delta : (\mathit{S}')_{\Delta}^{\infty} \neq \emptyset \}$, we will prove the lemma by showing that $\{\Delta : (\mathit{S}')_{\Delta}^{\infty} \neq \emptyset \}$ is bounded by $\frac{1}{2q'}$ and then showing that $\frac{1}{2q'}$ bound is in fact inside the set.

It follows from Lemma \ref{max_shrink} that for any $\Delta \in \{\Delta : (\mathit{S}')_{\Delta}^{\infty} \neq \emptyset \}$, there exists $x \in (\mathit{S}')_{\Delta}^{\infty}$ such that
\begin{equation*}
    x_m \geq \Delta - \frac{1}{2q \beta_{T,m}} \ \forall m \in [M] \ \text{, and} \ \sum_{m=1}^{M}\frac{\beta_{T,m}}{c_m} x_m \leq \frac{1}{2} - \Delta q'.
\end{equation*}
Therefore, we can combine these two inequalities to obtain the following:
\begin{equation}\label{delta_bound}
    \begin{split}
       \Delta & \leq \frac{\frac{1}{2} - \sum_{m=1}^{M}\frac{\beta_{T,m}}{c_m}x_m}{q'}\\ 
        & \leq \frac{\frac{1}{2} - \sum_{m=1}^{M}\frac{\beta_{T,m}}{c_m}(\Delta - \frac{1}{2q \beta_{T,m}})}{q'}\\ 
       & = \frac{\frac{1}{2} + \sum_{m=1}^{M}\frac{(\frac{1}{2q} - \Delta\beta_{T,m})}{c_m}}{q'} \\
       &= \frac{\frac{1}{2} + \frac{1}{2q}\sum_{m=1}^{M}\frac{1}{c_m} -\Delta \sum_{m=1}^{M}\frac{\beta_{T,m}}{c_m}}{q'} \\
       & = \frac{1 - \Delta q'} {q'}, 
    \end{split}
\end{equation}
where the last equality uses the fact that $\sum_{m=1}^M \frac{1}{c_m} = q$ and $\sum_{m=1}^{M} \frac{\beta_{T,m}}{c_m} = q'$.
Rearranging \eqref{delta_bound}, we get
\begin{equation*}
    \Delta \leq \frac{1 - \Delta q'} {q'} \quad \Longleftrightarrow \quad \Delta \leq \frac{1}{2 q'},
\end{equation*}
and therefore, the set $\{\Delta : (\mathit{S}')_{\Delta}^{\infty} \neq \emptyset \}$ is bounded by $\frac{1}{2 q'}$.

Now, we show that $\frac{1}{2 q'}$ is in $\{\Delta : (\mathit{S}')_{\Delta}^{\infty} \neq \emptyset \}$.
To do so, we show that the point $w \in \mathbb{R}^m$, defined for all $m \in [M]$, as
\begin{equation*}
    w_m = \frac{1}{2 q'} - \frac{1}{2 q\beta_{T,m}},
\end{equation*}
is in $(\mathit{S}')_{\Delta}^{\infty}$ when $\Delta = \frac{1}{2 q'}$.
From Lemma \ref{max_shrink}, this is equivalent to both
\begin{enumerate}
    \item $w_m \geq \frac{1}{2 q'} - \frac{1}{2q \beta_{T,m}} \quad \forall m \in [M]$,
    \item $\sum_{m=1}^{M}\frac{\beta_{T,m}}{c_m} x_m \leq \frac{1}{2} - \frac{1}{2 q'} q'$.
\end{enumerate}
First, 1) holds by definition as
\begin{equation}
    w_m = \frac{1}{2 q'} - \frac{1}{2 q\beta_{T,m}} \geq \frac{1}{2 q'} - \frac{1}{2q \beta_{T,m}},
\end{equation}
for all $m \in [M]$.
Then, we show that 2) holds,
\begin{equation}
\begin{split}
    & \sum_{m=1}^{M}\frac{\beta_{T,m}}{c_m} w_m \leq \frac{1}{2} - \frac{1}{2 q'} q'\\
    \Longleftrightarrow \quad & \sum_{m=1}^{M}\frac{\beta_{T,m}}{c_m} \left( \frac{1}{2 q'} - \frac{1}{2 q\beta_{T,m}} \right) \leq \frac{1}{2} - \frac{1}{2 q'} q'\\
    \Longleftrightarrow \quad & \sum_{m=1}^{M}\frac{\beta_{T,m}}{c_m} \frac{1}{2 q'} - \sum_{m=1}^{M}\frac{1}{c_m} \frac{1}{2 q} \leq \frac{1}{2} - \frac{1}{2 q'} q'\\
    \Longleftrightarrow \quad & q' \frac{1}{2 q'} - q \frac{1}{2 q} \leq \frac{1}{2} - \frac{1}{2 q'} q'\\
    \Longleftrightarrow \quad & 0 \leq 0.\\
\end{split}
\end{equation}
Therefore, if $\Delta = \frac{1}{2 q'}$, the set $(\mathit{S}')_{\Delta}^{\infty}$ is nonempty.
It follows that $\frac{1}{2 q'}$ is in $\{\Delta : (\mathit{S}')_{\Delta}^{\infty} \neq \emptyset \}$.

Finally, since $\{\Delta : (\mathit{S}')_{\Delta}^{\infty} \neq \emptyset \}$ is bounded by $\frac{1}{2 q'}$ and also contains $\frac{1}{2 q'}$, it holds that 
\begin{equation*}
    \mathcal{H}_{\mathit{S}}^{\infty} = \sup\{\Delta : (\mathit{S}')_{\Delta}^{\infty} \neq \emptyset \} = \frac{1}{2 q'}.
\end{equation*}
This completes the proof.
\end{proof}

Then, we prove Lemma \ref{transformed_sharp} in the following.


\begin{proof}
    The sharpness of the simplex can be written as
    \begin{equation}
        \mbox{Sharp}_{\mathit{S}'}^{\infty} ({\Delta}) = \underset{\mathbf{x} \in \mathit{S}'}{\max} \,\, \underbrace{\underset{\mathbf{y} \in (\mathit{S}')^{\infty}_{\Delta}}{\min}\,||\mathbf{x}-\mathbf{y}||_2}_{  h(x)}.
    \end{equation}
    Since $||\mathbf{x}-\mathbf{y}||_2$ is convex in $(x,y)$, it follows that $h(x)$ is convex.
    Then, since the maximum of a convex function over a polytope is attained at a vertex, it holds that
    \begin{equation}\label{max_dis}
        \mbox{Sharp}_{\mathit{S}'}^{\infty} ({\Delta}) = \max_{v \in V'} h(v),
    \end{equation}
    where $V'$ is the set of vertices of $S'$.
    Therefore, we will study the sharpness by studying $h(v)$ for all $v \in V'$.
    To do so, we first note that since $S'$ is a simplex, there are $M + 1$ vertices and each vertex, denoted by $v_m$ for $m \in \{0,1, ...,M\}$, can be written as $v_m = ((\mathbf{A}')^{v_m})^{-1} \mathbf{b}^{v_m}$, where $(\mathbf{A}')^{v_m}$ denotes $\mathbf{A}'$ when the row $m$ is removed and $\mathbf{b}^{v_m}$ denotes $\mathbf{b}$ (defined in \eqref{b}) with element $m$ removed.
    Also, since $(\mathit{S}')^{\infty}_{\Delta}$ is a simplex, we denote each vertex as $u_m = ((\mathbf{A}')^{u_m})^{-1} \tilde{\mathbf{b}}^{u_m}$, for $m \in \{0,1, \cdots,M\}$, where $(\mathbf{A}')^{u_m} = (\mathbf{A}')^{v_m}$ and $\tilde{\mathbf{b}}^{u_m}$ denotes the vector defined by $\tilde{b}_j = b_j - 
    \Delta||\mathbf{a}'_{j}||_1$ for all $j \in \{0,...,m-1,m+1,...,M\}$, where the $m$th element has been removed.
    It is worthwhile to note that these definitions imply that, for every $m \in \{0,1, ...,M\}$,
    \begin{equation}
        \label{mequate}
        \begin{split}
            (v_m - u_m) & = ((\mathbf{A}')^{v_m})^{-1} \mathbf{b}^{v_m} - ((\mathbf{A}')^{u_m})^{-1} \tilde{\mathbf{b}}^{u_m}\\
            & = ((\mathbf{A}')^{v_m})^{-1}( \mathbf{b}^{v_m} - \tilde{\mathbf{b}}^{u_m})\\
            & = \Delta ((\mathbf{A}')^{v_m})^{-1} \alpha^{v_m},
        \end{split}
    \end{equation}
    where $\alpha = [||\mathbf{a}'_{0}||_1\ \cdots \ ||\mathbf{a}'_{M}||_1]$ and $\alpha^{v_m}$ denotes the vector $\alpha$ with the element $m$ removed. 

    In order to study $h(v)$, we characterize the points in $(\mathit{S}')^{\infty}_{\Delta}$ that attain the minimum in the definition of $h(v)$ (i.e., the $y \in (\mathit{S}')^{\infty}_{\Delta}$ such that $h(v) = ||v-y||_2$).
    Specifically, we will show that for the vertex $v_m$, the point in  $(\mathit{S}')^{\infty}_{\Delta}$ that attains the minimum is $u_m$.
    Precisely, we need to show that $v_m - u_m$ is in the normal cone of $u_m$ for all $m \in \{ 0,...,M\}$.
    In particular, using standard results from convex analysis (i.e. Theorem 6.46 in \cite{rockafellar2009variational}), we need that there exists $z \in \mathbb{R}^M_+$ such that
    \begin{equation}
    \label{norm_cone}
        (v_m - u_m) = ((\mathbf{A}')^{v_m})^T z.
    \end{equation}
    Using \eqref{mequate}, we find such a $z$ by showing that
    \begin{equation}\label{z}
         z = \Delta (((\mathbf{A}')^{v_m})^T)^{-1}((\mathbf{A}')^{v_m})^{-1} \alpha^{v_m} \in \mathbb{R}_+.
    \end{equation}




In particular, we show that this holds for vertex $v_m$ for $m = 0$ and $m \neq 0$ in the following.
\begin{itemize}
    \item  For $m = 0$, 
\begin{equation*}
    \alpha^{v_0} = [\beta_{T,1},\beta_{T,2},\cdots,\beta_{T,M}]^T,
\end{equation*}
and for the diagonal matrix $(\mathbf{A}')^{v_0}$,
\begin{equation}
\label{matt1}
\begin{split}
     (((\mathbf{A}')^{v_0})^T)^{-1} & =  ((\mathbf{A}')^{v_0})^{-1} \\ & = \begin{bmatrix}
   -\frac{1}{\beta_{T,1}} & 0 & \cdots & 0 \\ 0 & -\frac{1}{\beta_{T,2}} &  & \vdots \\ \vdots  &  & \ddots & 0  \\ 0 & \cdots & 0 & -\frac{1}{\beta_{T,M}}
\end{bmatrix}.
\end{split}
\end{equation}
 Then it follows from \eqref{z} that
 \begin{equation*}
z = 
\begin{bmatrix}
   \frac{\Delta}{\beta_{T,1}} \\  \frac{\Delta}{\beta_{T,2}} \\ \vdots \\\frac{\Delta}{\beta_{T,M}}   
\end{bmatrix} \in \mathbb{R}^M_+.
\end{equation*}

    \item For $ m\neq 0 $, 

\begin{equation*}
    \alpha^{v_m} = [q', \beta_{T,1},\cdots,\beta_{T,m-1},\beta_{T,m+1},\cdots,\beta_{T,M}]^T,
\end{equation*}
and

 \begin{equation}\label{inverse}
 \begin{split}
     & ((\mathbf{A}')^{\mathit{v}_m})^{-1} = \\ &
\begin{bmatrix}
  0 & \frac{-1}{\beta_{T,1}} & 0 & & \cdots &  & 0 \\
  \vdots & & \ddots & & & & \vdots \\
  0 & \cdots & 0 & \frac{-1}{\beta_{T,m-1}} & 0 & \cdots & 0 \\
  \rho_m & \frac{\rho_m}{c_1} & \cdots &\frac{\rho_m}{c_{m-1}} & \frac{\rho_m}{c_{m+1}} & \cdots & \frac{\rho_m}{c_M} \\
  0 & & \cdots &  & \frac{-1}{\beta_{T,m+1}} & \cdots & 0 \\
  \vdots & & & & & \ddots & \vdots \\
  0 &  &  & \cdots  &  &  & \frac{-1}{\beta_{T,M}}
\end{bmatrix}.
 \end{split}
 \end{equation}
 
 It follows that
 \begin{equation*}
      ((\mathbf{A}')^{\mathit{v}_m})^{-1}\alpha^{v_m}  = \left[-1, \cdots , -1 , 2q'\rho_m-1, \cdots ,-1 \right]^T
 \end{equation*}
and therefore,
\begin{equation*}
z = 
\begin{bmatrix}
   \rho_m\gamma \\  \frac{1}{\beta_{T,1}}+\frac{\rho_m}{c_1}\gamma \\ \vdots \\ \frac{1}{\beta_{T,m-1}}+\frac{\rho_m}{c_{m-1}}\gamma \\ \frac{1}{\beta_{T,m+1}}+\frac{\rho_m}{c_{m+1}}\gamma \\ \vdots \\ \frac{1}{\beta_{T,M}}+\frac{\rho_m}{c_M}\gamma  
\end{bmatrix} \in \mathbb{R}^M_+,
\end{equation*}
where $\gamma = 2q'\rho_m-1 > 0 $.


\end{itemize}

Thus far, we have shown that the minimum distance from each vertex to the shrunk version is attained at the corresponding vertex of the shrunk version and therefore
\begin{equation}\label{max_dis2}
    \begin{split}
        \mbox{Sharp}_{\mathit{S}'}^{\infty} ({\Delta}) & = \max_{m \in \{0,1,...,M\}} \| u_m - v_m \|_2,\\
        & = \max_{m \in \{0,1,...,M\}} \| \Delta ((\mathbf{A}')^{v_m})^{-1} \alpha^{v_m} \|_2\\
        & \begin{aligned}= \max\Big\{ & \Delta \sqrt{M},\\
        & \max_{m \in [M]} \Delta \sqrt{(M-1)+ ( 2q'\rho_m - 1)^2} \Big\}\end{aligned}\\
        & = \Delta \sqrt{(M-1)+ ( 2q'\tilde{\rho} - 1)^2},
    \end{split}
\end{equation}
where the third equality uses the form of $((\mathbf{A}')^{v_m})^{-1}$ given in \eqref{matt1} and \eqref{inverse}, and the last equality uses the fact that $q'\tilde{\rho} > 1$ for all $m$, where $\tilde{\rho} = \max_{m \in [M]} \rho_m$.
\end{proof}

Now that we state the exact sharpness of $\mathit{S}'$, the following lemma suggests that among various simplices, the sharpness of a symmetric simplex for a given positive $\Delta$ is the smallest, making them potentially more favorable for minimizing the regret.

\begin{lemma}\label{min_sharp}
   For a simplex $\mathit{S}'$ and for an arbitrary $0 \leq \Delta \leq \mathcal{H}_{\mathit{S}'}^{\infty}$ with vector $\boldsymbol{\rho} = \left[ \rho_1 , \rho_2 , \cdots , \rho_M \right]$, it holds:  
    \begin{equation}\label{min_sharp_eq}
       \underset{\boldsymbol{\rho}}{\min} \, \mbox{Sharp}_{S'}^{\infty} (\Delta) = \Delta \sqrt{(M-1) + (2M - 1)^2},
    \end{equation}
and 
\begin{equation}
    \underset{\boldsymbol{\rho}}{\argmin} \, \mbox{Sharp}_{\mathit{S}'}^{\infty}({\Delta}) = \{ \boldsymbol{\rho} = p\mathds{1} : \forall p \in \mathbb{R}^{+}\}.
\end{equation}
\end{lemma}

\subsubsection{Proof of Lemma \ref{min_sharp}}
\label{proof_min}

To prove Lemma \ref{min_sharp}, let us use the following lemma. 

\begin{lemma}\label{function}
    Consider the function $g:\mathbb{R}^M_{++} \rightarrow \mathbb{R}$ defined as
    \begin{equation*}
        g(\mathbf{b}) = \left(\max_{m \in [M]} b_m \right)\sum_{m=1}^{M}\frac{1}{b_m}.
    \end{equation*}
    Then, for any $\mathbf{b} \in \mathbb{R}^M_{++}$, it holds that $g(\mathbf{b}) \geq M$.
    Furthermore,  $g(\mathbf{b}) = M$ if and only if $\mathbf{b} = B\mathds{1}$ for some $B \in \mathbb{R}_{++}$.
\end{lemma}
\begin{proof}
    First, we show that $g(\mathbf{b}) \geq M$.
    Since $b_m > 0 $  for all $m \in [M]$, it holds that:
    \begin{equation}
    \begin{split}
        g(\mathbf{b}) & = \left(\max_{m \in [M]} b_m \right)\sum_{m=1}^{M}\frac{1}{b_m}\\
        & \geq \left(\max_{m \in [M]} b_m \right)\sum_{m=1}^{M}\frac{1}{\max_{m \in [M]} b_m}\\
        & = M.
    \end{split}
    \end{equation}
    Then, we show that
    \begin{equation*}
        g(\mathbf{b}) = M \quad \Longleftrightarrow \quad \exists B \in \mathbb{R}_{++} \text{ s.t. } \mathbf{b} = B\mathds{1}.
    \end{equation*}
\begin{itemize}
    \item ($\Longleftarrow$) If there exists $B \in \mathbb{R}_{++}$ such that  $\mathbf{b} = B\mathds{1}$, then
\begin{equation*}
    g(\mathbf{b}) = B\sum_{m=1}^{M}\frac{1}{B} = \sum_{m=1}^{M}\frac{B}{B}= M.
\end{equation*}
\item ($\Longrightarrow$) We prove the contrapositive. 
If there does not exist $B \in \mathbb{R}_{++}$ such that $\mathbf{b} = B\mathds{1}$, then $b_{m'} < \max_{m \in [M]} b_m$ for at least one $m' \in [M]$.
It follows that,
\begin{equation}
\begin{split}
    g(\mathbf{b}) & = \left(\max_{m \in [M]} b_m \right)\sum_{m=1}^{M}\frac{1}{b_m}\\
    & > \left(\max_{m \in [M]} b_m \right)\sum_{m=1}^{M}\frac{1}{\max_{m \in [M]} b_m}\\
    & = M,
\end{split}
\end{equation}
and therefore, $g(\mathbf{b}) \neq M$.
\end{itemize}
\end{proof}

\subsubsection{Completing the proof of Lemma \ref{min_sharp}}

\begin{proof}
     Given sharpness of the transformed safe set $\mathit{S}'$ from Lemma \ref{transformed_sharp} (Eq. \eqref{eq_sharp}), we can rewrite it as follows: 
     \begin{equation}
     \begin{split}
         \mbox{Sharp}_{\mathit{S}'}^{\infty} ({\Delta}) & = \Delta \sqrt{(M-1)+ ( 2q'\tilde{\rho} - 1)^2} \\ & =  \Delta \sqrt{(M-1)+ ( 2f(\boldsymbol{\rho}) - 1)^2}, 
         \end{split}
     \end{equation}
 where $ f(\boldsymbol{\rho})=\left(\max_{m \in [M]} \rho_m \right)\sum_{m=1}^{M}\frac{1}{\rho_m}.$ 
 Then the following inequalities hold for a constant M and positive $\Delta$, using Lemma \ref{function}: 
\begin{equation}\label{min of sharp}
    \begin{split}
    f(\boldsymbol{\rho}) & \geq M \\
        \Rightarrow \sqrt{(M-1)+(2f(\boldsymbol{\rho})-1)^2} & \geq \sqrt{(M-1)+(2M-1)^2} \\ &= \sqrt{4M^2 - 3M}\\
        \Rightarrow \Delta\sqrt{(M-1)+(2f(\boldsymbol{\rho})-1)^2} & \geq \Delta\sqrt{4M^2 - 3M},
\end{split}
\end{equation}
and $\mbox{Sharp}_{\mathit{S}'}^{\infty} ({\Delta}) = \Delta\sqrt{4M^2 - 3M} $ when $\boldsymbol{\rho} = p\mathds{1}$ for all $p \in \mathbb{R}^{+}$.
Then Lemma \ref{min_sharp} is proved. 
\end{proof}

\subsubsection{Regret Bound, a function of privacy levels}\label{bound_lim}

Since $\beta_{t,m}$ is a function of $\alpha_m$ \eqref{beta},
we define vector $\mathit{a}=[\alpha_1,\alpha_2,...,\alpha_M]$, called privacy vector which is a vector of all $M$ agents' privacy level parameters (defined in \eqref{privacy_level}). Then, using Theorem \ref{main_2} and Lemma \ref{transformed_sharp}, the 
regret upper bound for Safe-Private Lin-UCB with a simplex safe set $\mathcal{S}$ can be written as a function of privacy vector $\mathit{a}$ . 

\begin{equation}\label{R2-Simp}
    \begin{split}
        R_T & \leq r(T, \mathit{a}) \\ & := 2LKS\sqrt{M}T'\\ &+ L(T-T')\frac{2\sqrt{2}K\tilde{\beta}_T}{\sqrt{2\nu + \check{\lambda}T'}}\sqrt{(M-1)+(2q'\tilde{\rho}-1)^2}\\
         & +L\max(\frac{1}{2q'},2) \\ & \times \sqrt{2d\log(1 + \frac{TK^2}{d\nu})(T-T')\sum_{m \in [M]}\beta_{T,m}^2},
     \end{split}
\end{equation} 
where $T' = \max(t'_h , t_{\delta'}, \frac{2}{\check{\lambda}}(\tilde{\beta}_TT)^{\frac{2}{3}})$ and $\tilde{\beta}_T = \underset{m \in [M]}{\max\beta_{T,m} }$.

Now we can state Lemma \ref{limit_lemma}.

\begin{lemma}\label{limit_lemma}
    Let given  $r(\mathit{a}) = \underset{T \rightarrow \infty}{\lim}\frac{r(T, \mathit{a})}{\left(T^{2}(\log T)\right)^{\frac{1}{3}}}$. Then for $\tilde{\alpha}= \underset{m \in [M]}{\max} \alpha_m $, it holds: 

\begin{equation}\label{limit_value_app}
\begin{split}
 r(\mathit{a}) 
        & =  2LK  \left(d(R^2 +{\tilde{\alpha}}^2\sigma^2)\right)^{\frac{1}{3}} \\
    &   \times    \left(\frac{2S\sqrt{M}}{\check{\lambda}} + \sqrt{(M-1)+ ( 2f(\mathit{a}) - 1)^2} \right),  \\
    \end{split}
\end{equation}
where the function $f(\mathit{a})$ is given by:
\begin{equation}\label{fa_app}
    f(\mathit{a}) = \max_{m \in [M]} \sum_{m'=1}^{M} \frac{c_{m}\sqrt{(R^2 +\alpha_{m'}^2\sigma^2)}}{c_{m'}\sqrt{(R^2 +\alpha_{ m}^2\sigma^2)}}.
\end{equation}
\end{lemma}
\subsubsection{Proof of Lemma \ref{limit_lemma}}
\label{r_apn}

\begin{proof}
 
Given Equation \eqref{R2-Simp}, 
\begin{equation}
    \begin{split}
        r(T, \mathit{a}) & = 2LKS\sqrt{M}T'\\ &+ L(T-T')\frac{2\sqrt{2}K\tilde{\beta}_T}{\sqrt{2\nu + \check{\lambda}T'}}\sqrt{(M-1)+(2q'\tilde{\rho}-1)^2}\\
         & +L\max(\frac{1}{2q'},2) \\ & \times \sqrt{2d\log(1 + \frac{TK^2}{d\nu})(T-T')\Bigl(\!\!\sum_{m \in [M]}\!\!\!\beta_{T,m}^2\Bigl)},
     \end{split}
\end{equation} 
where $T' = \max(t'_h , t_{\delta'}, \frac{2}{\check{\lambda}}(\tilde{\beta}_TT)^{\frac{2}{3}})$ and $\tilde{\beta}_T = \underset{m \in [M]}{\max\beta_{T,m} }$.
Also, note that
 \begin{equation}\label{alpha_hat}
\begin{split}
      \tilde{\beta}_T & = \underset{m \in [M]}{\max\beta_{T,m}} \\ & =  \underset{m \in [M]}{\max }\sqrt{d(R^2 +\alpha_m^2\sigma^2)\log{(\frac{1+ TK^2/\nu}{\delta'/M})}} + S\sqrt{\nu} \\ 
      & = \sqrt{d(R^2 +  (\underset{m \in [M]}{\max}\alpha_m)^2\sigma^2)\log{(\frac{1+ TK^2/\nu}{\delta'/M})}} + S\sqrt{\nu} \\ 
    & = \sqrt{d(R^2 +  \tilde{\alpha}^2 \sigma^2)\log{(\frac{1+ TK^2/\nu}{\delta'/M})}} + S\sqrt{\nu}. \\
\end{split}
\end{equation}

Therefore, it holds that:
\begin{equation}\label{lim}
    \begin{split}
        r(\mathit{a}) & =  \underset{T \rightarrow \infty}{\lim}\frac{r(T, \mathit{a})}{\left(T^{2}(\log T)\right)^{\frac{1}{3}}}\\
        & = \underset{T \rightarrow \infty}{\lim}  \frac{1}{T^{2/3}(\log(T))^{\frac{1}{3}}}\Biggl(2LKS\sqrt{M}T'\\ & + L(T-T')\frac{2\sqrt{2}K\tilde{\beta}_T}{\sqrt{2\nu + \check{\lambda}T'}}\sqrt{(M-1)+(2q'\tilde{\rho}-1)^2}\\
        & + L\max(\frac{1}{2q'}, 2)\\ & \times\sqrt{2d\log(1 + \frac{TK^2}{d\nu})(T-T')\Bigl(\!\!\sum_{m \in [M]}\!\!\!\beta_{T,m}^2\Bigl)} \Biggl)\\ = & \underset{T \rightarrow \infty}{\lim}  \frac{2LKS\sqrt{M}(\frac{2}{\check{\lambda}}(\tilde{\beta}_TT)^{\frac{2}{3}})}{T^{2/3}(\log(T))^{\frac{1}{3}}} + \\ &\underset{T \rightarrow \infty}{\lim}  \frac{1}{T^{2/3}(\log(T))^{\frac{1}{3}}} \\& \times (\frac{2\sqrt{2}LKT\tilde{\beta}_T\sqrt{(M-1)+(2q'\tilde{\rho}-1)^2} }{\sqrt{2\nu + \check{\lambda}(\frac{2}{\check{\lambda}}(\tilde{\beta}_TT)^{\frac{2}{3}})}})
        \\ = & \frac{4LKS\sqrt{M}}{\check{\lambda}} \\
         & \times\!\!\Biggl(\!\underset{T \rightarrow \infty}{\lim}\!\!\Bigl(\!\!\sqrt{\frac{d(R^2 +\tilde{\alpha}^2\sigma^2)\log{(\frac{1+TK^2/\nu}{\delta'/M})}}{\log(T)}}\!+\!\frac{S\sqrt{\nu}}{\sqrt{\log(T)}}\Bigl)^{\!\frac{2}{3}}\!\!\Biggl) \\ & +\underset{T \rightarrow \infty}{\lim} \frac{2LK}{T^{2/3}log(T)^{\frac{1}{3}}}(\frac{T\tilde{\beta}_T\sqrt{(M-1)+(2q'\tilde{\rho}-1)^2} }{\sqrt{\nu+(T\tilde{\beta}_T)^{\frac{2}{3}}}})
        \\ = & \frac{4LKS\sqrt{M}}{\check{\lambda}}d^{\frac{1}{3}}(R^2+\tilde{\alpha}^2\sigma^2)^{\frac{1}{3}}+ 2LKd^{\frac{1}{3}}(R^2+\tilde{\alpha}^2\sigma^2)^{\frac{1}{3}} \\
        & \times \underset{T \rightarrow \infty}{\lim} \frac{(T\tilde{\beta}_T)^{\frac{1}{3}}(\sqrt{(M-1)+(2q'\tilde{\rho}-1)^2})}{\sqrt{\nu+(T\tilde{\beta}_T)^{\frac{2}{3}}}}
        \\ = & \frac{4LKS\sqrt{M}}{\check{\lambda}}d^{\frac{1}{3}}(R^2+\tilde{\alpha}^2\sigma^2)^{\frac{1}{3}} + \\
        & 2LKd^{\frac{1}{3}}(R^2+\tilde{\alpha}^2\sigma^2)^{\frac{1}{3}} \underset{T \rightarrow \infty}{\lim} (\sqrt{(M-1)+(2q'\tilde{\rho}-1)^2}).
    \end{split}    
\end{equation}
Given  $\tilde{\rho} = \underset{m \in [M]}{\max} \rho_m =  \underset{m \in [M]}{\max}\frac{c_m}{\beta_{T,m}}$, then it holds:

\begin{equation}
\label{lim_min}
\begin{split}
    & \underset{T \rightarrow \infty}{\lim} \sqrt{(M-1)+\left(2q'\tilde{\rho}-1\right)^2} \\
    = & \underset{T \rightarrow \infty}{\lim} \sqrt{(M-1)+\left(2\tilde{\rho}\sum_{m=1}^{M} \frac{1}{\rho_m}-1\right)^2} \\
    = & \underset{T \rightarrow \infty}{\lim} \sqrt{(M-1)+\left(2\tilde{\rho}\sum_{m=1}^{M} \frac{\beta_{T,m}}{c_m}-1\right)^2} \\
    = & \underset{T \rightarrow \infty}{\lim} \sqrt{(M-1)+\left(2\big(\underset{m \in [M]}{\max}\frac{c_m}{\beta_{T,m}}\big)\sum_{m=1}^{M} \frac{\beta_{T,m}}{c_m}-1\right)^2} \\
    = & \underset{T \rightarrow \infty}{\lim} \sqrt{(M-1)+\left(2\big(\underset{m \in [M]}{\max}\sum_{m'=1}^{M} \frac{c_m\beta_{T,m'}}{c_m'\beta_{T,m}}\big)-1\right)^2} \\
    = & \lim_{T \rightarrow \infty} \sqrt{(M-1)+\left(2\big(\underset{m \in [M]}{\max} \lim_{T \rightarrow \infty}\sum_{m'=1}^{M} \frac{c_m\beta_{T,m'}}{c_m'\beta_{T,m}}\big)-1\right)^2} \\
    = & \Biggl((M-1)+\biggl(2 \underset{m \in [M]}{\max} \underset{T \rightarrow \infty}{\lim} \\  \times  & \sum_{m'=1}^{M} \frac{c_{m}(\sqrt{d(R^2 +\alpha_{m'}^2\sigma^2)\log{( \frac{1+ TK^2/\nu}{\delta'/M})}} + S\sqrt{\nu})}{c_{m'}(\sqrt{d(R^2 +\alpha_{ m}^2\sigma^2)\log{( \frac{1+ TK^2/\nu}{\delta'/M})}} + S\sqrt{\nu})}-1\biggl)^2\Biggl)^{\frac{1}{2}}\\
    = & \sqrt{(M-1)+\left(2 \max_{m \in [M]} \sum_{m'=1}^{M} \frac{c_{m}\sqrt{(R^2 +\alpha_{m'}^2\sigma^2)}}{c_{m'}\sqrt{(R^2 +\alpha_{ m}^2\sigma^2)}}-1\right)^2}. 
\end{split}
\end{equation}
Therefore, for $f(\mathit{a}) = \max_{m \in [M]} \sum_{m'=1}^{M} \frac{c_{m}\sqrt{(R^2 +\alpha_{m'}^2\sigma^2)}}{c_{m'}\sqrt{(R^2 +\alpha_{ m}^2\sigma^2)}}$,
\begin{equation}\label{budget_r}
\begin{split}
    r(\mathit{a})  =& \frac{4LKS\sqrt{M}}{\check{\lambda}}d^{\frac{1}{3}}(R^2+\tilde{\alpha}^2\sigma^2)^{\frac{1}{3}} \\
        & + 2LKd^{\frac{1}{3}}(R^2+\tilde{\alpha}^2\sigma^2)^{\frac{1}{3}}\sqrt{(M-1)+(2f(\mathit{a})-1)^2}\\
        = & 2LK(d(R^2 +{\tilde{\alpha}}^2\sigma^2))^{\frac{1}{3}} \times \\
     & \left(\frac{2S\sqrt{M}}{\check{\lambda}} + \sqrt{(M-1)+ ( 2f(\mathit{a}) - 1)^2}\right).
\end{split}
\end{equation}


\end{proof}

\subsubsection{Completing the proof}
\label{compl_proof}

\begin{theorem}[Duplicate of Theorem \ref{main_3}]
    Consider the privacy vector $\mathit{a}^* \in \mathbb{R}^M$, for which the $m$th element is defined as
    \begin{equation*}
        \alpha^*_m = \sqrt{\left(\frac{R^2}{\sigma^2}  +\tilde{r}^2\right)\frac{c_m^2}{\Tilde{c}^2} - \frac{R^2}{\sigma^2}},
    \end{equation*}
    where $\tilde{c} = \max_{m \in [M]} c_m$ and
    \begin{equation*}
        \tilde{r} =\frac{1}{\sigma}\sqrt{\frac{U^3}{8 L^3 K^3 d M^{\frac{3}{2}} \left(\frac{2 S}{\hat{\lambda}} + \sqrt{4M - 3}\right)^3} - R^2}.
    \end{equation*}
    It holds that $\mathit{a}^* \in \mathcal{A}_*$.
\end{theorem}

\begin{proof}
In order to show that $\mathit{a}^* \in \mathcal{A}_*$, we need that both
\begin{enumerate}
    \item $r(\mathit{a}^*) = U$,
    \item $r(\mathit{a}^* + v) > U$ for any $v \in \mathbb{R}_+ \setminus \{ \mathbf{0} \}$.
\end{enumerate}
First, we show 1).
To do so, note that
\begin{align*}
    f(\mathit{a}^*) &= \max_{m \in [M]} \sum_{m'=1}^{M} \frac{c_{m}\sqrt{R^2 +(\alpha_{m'}^*)^2\sigma^2}}{c_{m'}\sqrt{R^2 +(\alpha_{ m}^*)^2\sigma^2}}\\
    & = \sum_{m'=1}^{M} \frac{\tilde{c}\sqrt{R^2 +\tilde{r}^2\sigma^2}}{\tilde{c}\sqrt{R^2 +\tilde{r}^2\sigma^2}} = M.
\end{align*}
Also, it holds that
\begin{align*}
    \tilde{\alpha}^* & = \max_{m \in [M]} \alpha_m^*\\
    & = \max_{m \in [M]} \sqrt{(\frac{R^2}{\sigma^2}  +\tilde{r}^2)(\frac{c_m}{\tilde c})^2 - \frac{R^2}{\sigma^2}}\\
    & = \sqrt{(\frac{R^2}{\sigma^2}  +\tilde{r}^2)(\max_{m \in [M]} \frac{c_m}{\tilde c})^2 - \frac{R^2}{\sigma^2}} = \tilde{r}.
\end{align*}
Therefore,
\begin{align*}
    r(a^*) & = 2LK(d(R^2 +{\tilde{r}}^2\sigma^2))^{\frac{1}{3}} 
     \left(\frac{2S\sqrt{M}}{\check{\lambda}} + \sqrt{4M^2-3M}\right)\\
     & = U.
\end{align*}
Then, we show 2), that $r(a) > U$ where $a = \mathit{a}^* + v$ for any $v \in \mathbb{R}_+ \setminus \{ \mathbf{0} \}$.
First, let $g(h(a)) = f (a)$, where $h : \mathbb{R}^m_{++} \rightarrow \mathbb{R}^m_{++}$,
\begin{equation}
\label{h_func}
    h_m ( \alpha_m) :=  \frac{c_m}{\sqrt{R^2 + \alpha_m^2 \sigma^2}},
\end{equation}
and note that there exists positive real $B'$ such that $h_m ( \alpha_m^*) = B'$ for all $m \in [M]$.
Then, to show that $r(a) > U$, we consider two cases:
\begin{enumerate}
    \item If there exists a positive real $B$ such that $h_m(\alpha_m) = B$ for all $m \in [M]$, then it holds that $B = h_m(\alpha_m) < h_m ( \alpha_m^*) = B'$ since there exists at least one $\bar{m} \in [M]$ such that $\alpha_{\bar{m}} > \alpha_{\bar{m}}^*$ and $h_m$ is strictly decreasing for all $m$.
    Therefore, we can use the fact that the inverse of a strictly-decreasing function is strictly-decreasing (and therefore $h_m^{-1}$ is strictly-decreasing) to get that
    \begin{align*}
        \max_{m \in [M]} \alpha_m & = \max_{m \in [M]} h_m^{-1} (B)\\
        & > \max_{m \in [M]} h_m^{-1} (B')\\
        & = \max_{m \in [M]} \alpha_m^*.
    \end{align*}
    Also, Lemma \ref{function} tells us that $f(a) = g(h(a)) \geq M = f(a^*)$.
    Then, since $r(a)$ is strictly-increasing with respect to $f(a)$ and $\max_{m \in [M]} \alpha_m$, it follows that $r(a) > r(a^*) = U$.
    \item If there does not exist a positive real $B$ such that $h_m(\alpha_m) = B$ for all $m \in [M]$, then Lemma \ref{function} tells us that $f(a) = g(h(a)) > M = f(a^*)$.
    Also, by definition, $\max_{m \in [M]} \alpha_m \geq \max_{m \in [M]} \alpha_m^*$.
    Then, since $r(a)$ is strictly-increasing with respect to $f(a)$ and $\max_{m \in [M]} \alpha_m$, it follows that $r(a) > r(a^*) = U$.
\end{enumerate}

\end{proof}

\subsection{Numerical Experiment}\label{app:Appen4}
In this section, we discuss the numerical simulation results of Section \ref{num_exp} in detail. We first outline the algorithm and problem parameters for each setup, explain the steps to solve the underlying optimization problem in our simulation, and then provide a detailed interpretation of the results. 

The simulations were performed in a 3-dimensional action space with $M = 3$ agents, over a time horizon of $T = 30,000$ steps. The fixed constraint set is simplex $\mathit{S} = \{ \mathbf{x} \in \mathbb{R}^M: \mathbf{A}\mathbf{x} \leq \mathbf{b} \}$, where
\begin{equation}\label{A_num}
  \mathbf{A} = \begin{bmatrix}
  1 & 4 & 2\\ -1 & 0 & 0 \\ 0 & -1 & 0 \\ 0 & 0 & -1
\end{bmatrix},
\end{equation}
and 
\begin{equation}
  \mathbf{b} = \begin{bmatrix}
  \frac{1}{2}\\ \frac{1}{14}\\ \frac{1}{14} \\\frac{1}{14}
\end{bmatrix}.
\end{equation}

A linear Lipschitz function -  $f(\mathbf{y}) = \mathbf{c}^T\mathbf{y}$ for some $\mathbf{c} \in \mathbb R^3$ - was chosen as the reward function so that the optimal value lies at one of the vertices of the convex simplex.

In our first experiment, to observe the effect of the constraint set sharpness, the bandit parameter $\Theta_*$ and vector $\mathbf{c} = [0.8, 0.1, 0.1]$ were selected such that the optimal point is located at the sharpest vertex while satisfying $||\theta_{*,m}||_2\leq \frac{1}{2}$ and $||\mathbf{x}_m||_2\leq 2$ for all $m \in \{1, 2, 3\}$. The bandit parameter for the first agent was $\theta_{*,1} = [ 0 , 0 ,\frac{1}{2}]$ and for the other two agents, $\theta_{*,2} = \theta_{*,3} = [  -\frac{1}{13} , -\frac{1}{13} , -\frac{1}{13}]$.    

We then chose three privacy levels (defined in \eqref{privacy_level}) as $\frac{1}{4}, \frac{1}{2},$ and $1$, which are proportional to constraints parameters (Corollary \ref{P_S_T}). The base privacy parameters were $(\epsilon , \delta) = (2, 0.9)$ and given that the constraint set results in a response sensitivity of $1$, the base variance of the differential privacy noise for each agent was set to $\sigma = \frac{1}{2}\sqrt{2\ln(\frac{1.25}{0.9})}$ \cite{dwork2014algorithmic , zhao2019reviewing}. We then used the six different permutations of the chosen privacy levels as the privacy vectors for each setup, to compare the effects of privacy allocation and the shape of the constraint set. 

Additionally, the Sub-Gaussian parameter $\mathit{R}$ was chosen to be negligible, with a value of $0.001$, and the confidence parameter $\delta'$ and the regularization constant $\nu$ for least-squares estimation were set to $0.01$ and $0.1$ respectively. 

Given the reward function $f(\mathbf{y}) = \mathbf{c}^T\mathbf{y}$ with $\mathbf{c} \geq 0$, the following optimization must be solved to determine the agents' actions at time $t$: 

\begin{equation*}
            \mathbf{X}_t = \underset{ (\mathbf{X},\Theta) \in \mathcal{D}_{t} \times \mathcal{C}_{t}}{\argmax} \left( \sum_{m=1}^{3} c_m\theta_m^{T}\mathbf{x}_m \right).
\end{equation*}
 
Let $\mathbf{a}_m$ be the $m$'th column of matrix $\mathbf{A}$ in \eqref{A_num}. Then, for the set $(\mathbf{X},\Theta)$ such that $(\Theta\mathbf{X}, \Theta) \in \mathcal{S} \times \mathcal{C}_{t}$, it holds that:
\begin{equation*}
\sum_{j=1}^{3} \mathbf{a}_m\theta_m^{T}\mathbf{x}_m \leq b_i \quad \forall \Theta \in \mathcal{C}_{t},
\end{equation*}
which is equivalent to:
\[
\underset{\Theta \in \mathcal{C}_{t}}{\max} \left(\sum_{j=1}^{3} \mathbf{a}_{i,m}\theta_m^{T}\mathbf{x}_m \right) \leq b_i \quad \forall i \in \{1,2,3,4\},
\]
and
\[
\sum_{j=1}^{3} \underset{\Theta \in \mathcal{C}_{t}}{\max} \left(\mathbf{a}_{i,m}\theta_m^{T}\mathbf{x}_m \right) \leq b_i \quad \forall i \in \{1,2,3,4\}.
\]
Given that:
\begin{equation*}
   \underset{\Theta \in \mathcal{C}_{t}}{\max} \left(\mathbf{a}_{i,m}\theta_m^{T}\mathbf{x}_m \right) = \mathbf{a}_{i,m}\hat{\theta}_m^{T}\mathbf{x}_m + |\mathbf{a}_{i,m}|\beta_{t,m}\|\mathbf{x_m}\|_{\mathbf{G}_{t,m}^{-1}},
\end{equation*}
for all $i \in \{1,2,3,4\}$, the following holds: 

\begin{equation*}
   \sum_{j=1}^{3}\left(\mathbf{a}_{i,m}\hat{\theta}_m^{T}\mathbf{x}_m + |\mathbf{a}_{i,m}|\beta_{t,m}\|\mathbf{x_m}\|_{\mathbf{G}_{t,m}^{-1}}\right)\leq b_i.
\end{equation*}

 Thus, $\mathcal{D}_{t}$ can be defined as:

 \begin{equation}
     \left\{\mathbf{X}: 
\begin{aligned}
    & \|\mathbf{x_m}\|_2 \leq 2 \\
    & \sum_{j=1}^{3} \left(\mathbf{a}_{i,m} \hat{\theta}_m^{T}\mathbf{x}_m + |\mathbf{a}_{i,m}|\beta_{t,m}\|\mathbf{x_m}\|_{\mathbf{G}_{t,m}^{-1}} \right)\!\!\leq b_i \quad \!\!\!\!\forall i 
\end{aligned} 
\right\}.
 \end{equation}
 This definition was used to set the constraint for optimization problem in our simulation.
 
It follows that:

\begin{equation*}
\begin{aligned}
    \underset{ (\mathbf{X},\Theta) \in \mathcal{D}_{t} \times \mathcal{C}_{t}}{\max}& \left( \sum_{m=1}^{3} c_m\theta_m^{T}\mathbf{x}_m \right) \\ = \underset{ \mathbf{X} \in \mathcal{D}_{t}}{\max}&\left(\sum_{m=1}^{3} c_m \left(\underset{\Theta \in \mathcal{C}_{t}}{\max} \theta_m^{T}\mathbf{x}_m\right) \right) \\
    = \underset{ \mathbf{X} \in \mathcal{D}_{t}}{\max}&\left(\sum_{m=1}^{3} c_m \left(\hat{\theta}_m^{T}\mathbf{x}_m + \beta_{t,m}\|\mathbf{x_m}\|_{\mathbf{G}_{t,m}^{-1}}\right) \right) \\
    = \underset{ \mathbf{X} \in \mathcal{D}_{t}}{\max}&\left(\sum_{m=1}^{3} c_m \left(\hat{\theta}_m^{T}\mathbf{x}_m + \beta_{t,m}\|\mathbf{G}_{t,m}^{-\frac{1}{2}}\mathbf{x_m}\|_{2}\right) \right).
\end{aligned}
\end{equation*}

Relaxing the weighted norm to the infinity norm \cite{dani2008stochastic}, the non-convex optimization can be solved by solving $2dM$ convex optimizations, given that \( \|\mathbf{x}\|_{\infty} = \underset{i \in [d], \gamma \in \mathcal{E}}{\max} (\gamma x_i)\) \cite{cassel2022rate}, where $\mathcal{E} = \{-1,1\}$. Therefore,

\begin{equation*}
    \begin{aligned}  
    \underset{ \mathbf{X} \in \mathcal{D}_{t}}{\max}&\left(\sum_{m=1}^{3} c_m \left(\hat{\theta}_m^{T}\mathbf{x}_m +\beta_{t,m}\|\mathbf{G}_{t,m}^{-\frac{1}{2}}\mathbf{x_m}\|_{\infty}\right)\!\!\right)\!\!=\\ 
    \underset{\begin{array}{c}
        \gamma_m\in \mathcal{E} \\
        i_m \in [d] \\
        \forall m \in [M] \\     
    \end{array}}{\max}\!\!\!\underset{ \mathbf{X} \in \mathcal{D}_{t}}{\max}&\left(\sum_{m=1}^{3} c_m \left(\hat{\theta}_m^{T}\mathbf{x}_m +\gamma_m\beta_{t,m}(\mathbf{G}_{t,m}^{-\frac{1}{2}})_{i_m}\mathbf{x_m}\right) \!\!\right)
\end{aligned}
\end{equation*}
where $i_m$ is the row index of $\mathbf{G}_{t,m}^{-\frac{1}{2}}$.

Figure \ref{setup0_app} shows the regret performance of different privacy vectors when the optimal value lies at the sharpest corner of the constraint set. 


In all setups, the regret increases as the privacy level of the agent aligned with the optimal value rises. However, when this agent is the most private agent, the privacy allocations of the other two agents also affect the performance of the Safe-Private Lin-UCB algorithm. Therefore, a privacy allocation aligned with the constraint parameters can ensure better performance, which is consistent with the statement
of Theorem \ref{main_3}.
\begin{figure}[H]
 \centering
 \includegraphics[width=0.4\textwidth]{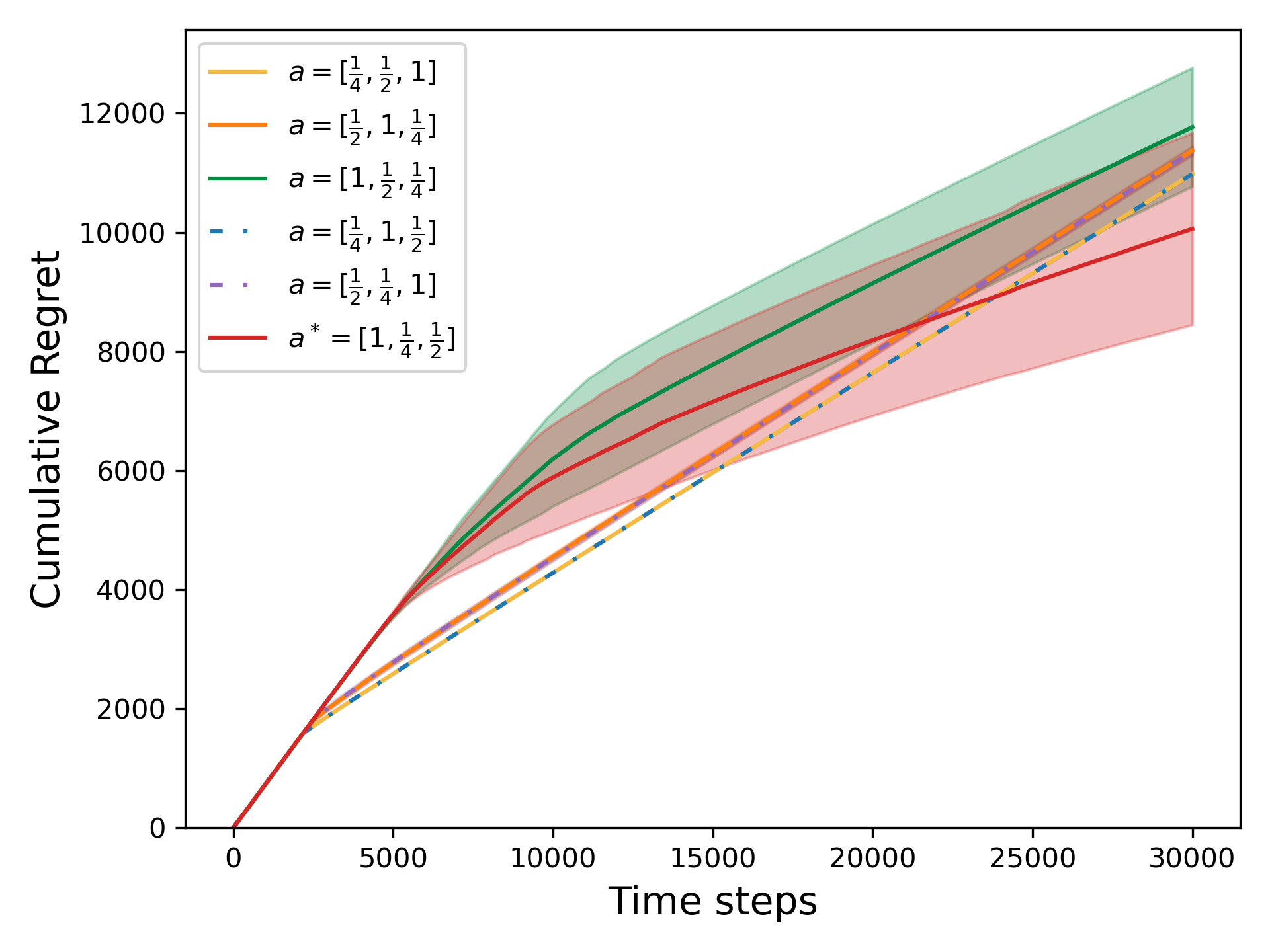}
 \caption{Cumulative regret for different privacy vectors when the optimal point lies at the sharpest vertex of the constraint set. The privacy vector $\mathit{a}^*$, as suggested in Theorem \ref{main_3}, shows better performance in terms of regret.}
 \label{setup0_app}
 \end{figure}

Additionally, since the optimal value can generally lie anywhere within the constraint set, we repeated the experiment with optimal points at different corners of the safe set $\mathcal{S}$. For one setup, the reward function vector is $\mathbf{c} = [0.1, 0.8, 0.1]$ with $\theta_{*,2} = [ 0 , 0 ,\frac{1}{2}]$ and $\theta_{*,1} = \theta_{*,3} = [  -\frac{1}{13} , -\frac{1}{13} , -\frac{1}{13}]$, while for the other setup, $\mathbf{c} = [0.1, 0.1, 0.8]$, with $\theta_{*,3} = [ 0 , 0 ,\frac{1}{2}]$ and $\theta_{*,1} = \theta_{*,2} = [  -\frac{1}{13} , -\frac{1}{13} , -\frac{1}{13}]$. 

Studying the performance of each privacy vector in all of these three setups can give us a better understanding of their overall performance, given that the learning process of the optimal value for each setup is impacted by the privacy level of the agents' responses aligned with the optimal value in that setup. Since the optimal values vary between the setups, to compare the regret, we first calculated the normalized regret for each setup, then averaged regret over the three setups that had the same privacy vectors but were different in optimal points. Figure \ref{norm_app} shows that $\mathit{a}^*$, as suggested in Theorem \ref{main_3}, performed better. Furthermore, among other privacy vectors, the worst case occurred when the privacy allocation order was exactly opposite to the constraint orders.

\begin{figure}[H]
 \centering
\includegraphics[width=0.4\textwidth]{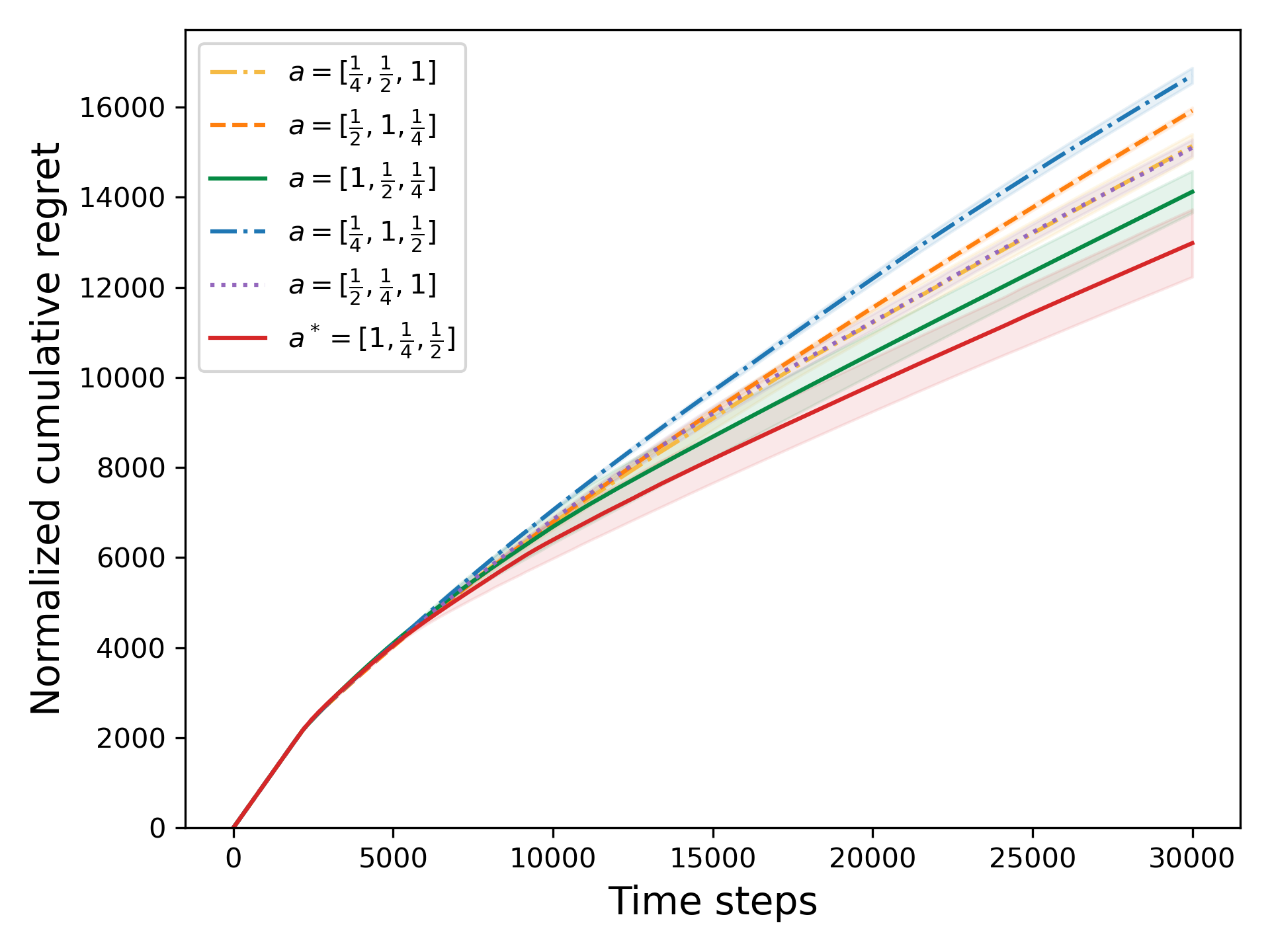}
 \caption{Average normalized cumulative regret over 3 different optimal point placement for each privacy vectors. The privacy vector $\mathit{a}^*$, as suggested in Theorem \ref{main_3}, shows better performance in terms of regret.}
 \label{norm_app}
 \end{figure}

\end{document}